\documentclass{article}
\usepackage[utf8]{inputenc}
\usepackage{amsmath,amssymb,amsfonts,bm,comment,amsthm}
\usepackage{graphicx,color}
\usepackage{tikz,pgfplots,tikz-network}
\usetikzlibrary{graphs,graphs.standard}
\usetikzlibrary{arrows.meta}
\usepackage{url}

\newcommand{\mathc}{}

\renewcommand{\hat}{\widehat}
\newcommand{\diag}{\hbox{\rm diag}}
\newcommand{\trace}{\hbox{\rm trace}}
\newcommand{\ones}{\bm 1}
\newcommand{\pig}{\boldsymbol{\pi}}
\newcommand{\bd}{\bm{d}}
\newcommand{\be}{\bm{e}}
\newcommand{\bv}{\bm{v}}
\newcommand{\bw}{\bm{w}}
\newcommand{\bu}{\bm{u}}

\newtheorem{proposition}{Proposition}
\newtheorem{corollary}{Corollary}

\newtheorem{Remark}{Remark}
\newtheorem{theorem}{Theorem}

\newtheorem{example}{Example}

\title{Cut-edge centralities in an undirected graph}

\author{Dario A.~Bini\thanks{Dipartimento di Matematica, Universit\`a di Pisa, Italy ({\tt dario.bini@unipi.it}).}\and 
	Steve Kirkland\thanks{Department of Mathematics, University of Manitoba,
		Winnipeg,
		MB, Canada
		({\tt stephen.kirkland@umanitoba.ca}).}
	\and 
	Guy Latouche\thanks{D\'epartment d'Informatique, Universit\'e Libre de Bruxelles, Belgium ({\tt guy.latouche@ulb.be}).} \and
	Beatrice~Meini\thanks{Dipartimento di Matematica, Universit\`a di Pisa, 
		Italy ({\tt beatrice.meini@unipi.it}).}}

\pgfplotsset{compat=1.18} 

\begin{document}
	
	\maketitle
	
	\begin{abstract}
		A centrality measure of the cut-edges of an undirected graph, given in [Altafini et al.~SIMAX 2023] and based on Kemeny's constant, is revisited. A numerically more stable expression is given to compute this measure, and an explicit expression is provided for some classes of graphs, including one-path graphs and trees formed by three or more branches. These results theoretically confirm the good physical behaviour of this centrality measure, experimentally observed in [Altafini et al.~SIMAX 2023]. Numerical tests are reported to check the stability and to confirm the good physical behaviour.
	\end{abstract}


\section{Introduction}
Let $G=(V,E)$ be an undirected graph, where $V$ is the set of vertices and $E\subset V\times V$ is the set of edges.  
In the paper \cite{abcmp}, Altafini et al.~introduced a centrality measure  to estimate the importance of an edge  $e\in E$ based on the variation of Kemeny's constant.
Informally, Kemeny's constant of the graph $G$, denoted by $\kappa(G)$, represents the average travel time in the network represented by the graph~\cite{ck:11}.

Motivated by this property, the importance of the edge $e$ is defined in \cite{abcmp}  as
 the difference $\kappa(\widetilde {\mathc G})-\kappa(G)$ between Kemeny's constant for the graph $\widetilde {\mathc G}$ obtained by removing the edge, and  Kemeny's constant $\kappa(\mathc G)$ of the original graph. 
This idea of analyzing the variation of  Kemeny's constant was previously proposed by Crisostomi et al.~in \cite{ck:11}. To overcome the Braess paradox
\cite{ck:11}, \cite{braess}, where the removal of an edge may lead to a decrease in  Kemeny's constant, the authors of \cite{abcmp} proposed to replace the graph $\widetilde{\mathc G}$ with the graph $\widehat {\mathc G}$ obtained by adding to $\widetilde {\mathc G}$ two loops at the two vertices of the edge that is removed. They  proved that  Kemeny's constant of the graph obtained this way can only increase. Relying on this fact, the centrality of an edge $e$ is defined as the difference $c(e)=\kappa(\widehat G)- \kappa(G),$ which is always nonnegative.
This definition is not valid if $e$ is a cut-edge (bridge) that is, an edge whose removal  disconnects the graph.

To deal with cut-edges, the same authors introduced a ``filtered'' and ``regularized'' version of this measure. Even though not stated explicitly in \cite{abcmp}, as we will prove later on, the centrality measure of the cut-edge $e$ in a graph $\mathc G$ essentially coincides with the difference
\begin{equation}\label{eq:0}
c(e)=\kappa(\mathc G)-\kappa(\hat{\mathc G}_1)-\kappa(\hat{\mathc G}_2),
\end{equation}
where  $\hat{\mathc G}_1$ and $\hat{\mathc G}_2$ are the two graphs obtained by removing the edge $e$ and adding two loops at  the endpoints of $e$. 
 This expression avoids the introduction of a regularization parameter and overcomes the problem of numerical cancellation intrinsic in the approach of \cite{abcmp}.

It was observed experimentally in \cite{abcmp}, relying on artificial tests and real road networks, that the centrality of cut-edges is related to the cardinality of the sets of vertices of the two connected components obtained after removal of the cut-edge, so that a cut-edge turns out to be more important if the ratio of the cardinalities of the two subgraphs $\hat{\mathc G}_1$ and $\hat{\mathc G}_2$ is closer to 1. 

In this contribution, we continue this analysis and formally prove this latter property for some graphs, including one-path graphs and trees formed by three (or more) branches. In particular, 
relying on \eqref{eq:0}, we provide explicit expressions of the edge centrality measure for these classes of graphs that yield simple algorithms for their numerical computation.

As a first result, we prove that the random walk on the graphs $\hat{\mathc G}_1$ and $\hat{\mathc G}_2$ is nothing  but the random walk on the original graph $\mathc G$ observed on the set of vertices of $\hat{\mathc G}_1$ and $\hat{\mathc G}_2$, respectively. This result provides a physical meaning of the definition of $c(e)$.
Moreover, we prove lower and upper bounds on $c(e)$, namely $\frac1{1-\lambda_1}\le c(e)\le \frac1{1-\lambda_{n-1}}$, where $-1\le\lambda_1\le\cdots\le \lambda_{n-1}<\lambda_n=1$ are the eigenvalues of the stochastic matrix $P$ defining the random walk on $\mathc G$.

Then we consider  one-path graphs, for which that adjacency matrix is an $n\times n$ tridiagonal matrix, and provide an explicit expression for the centrality $c(e)$ of the edge connecting vertex $m$ to vertex $m+1$ for $m=1,\ldots,n-1$. It turns out that $c(e)$ is essentially proportional to $m(n-m)$, where $n$ is the overall number of vertices in the graph. This explains why the central edges where $m$ is close to $n/2$ are more important than the peripheral edges where $m$ is close either to 1 or to $n$,
as experimentally observed in the paper \cite{abcmp}.

We extend this result to more general classes of trees. In particular, we provide an explicit expression of $c(e)$ for trees formed by three branches, and outline how to generalize this result to the case of any number  of branches. To do that, the explicit expression for Kemeny's constant given in \cite{kirk:ela} for graphs with no loops, is extended to the more general case of graphs having a loop at a given vertex.

As a byproduct of this research, we widen the class of graphs for which  Kemeny's constant can be explicitly given, see for instance \cite{breen19,breen22,breen23,breen_arxiv,faught2021,KD}.

The paper is organized in the following way. Section \ref{sec:prel} contains some preliminary results, where we recall the main definitions and properties of graphs, together with the definition and some properties of Kemeny's constant.

Section \ref{sec:tree} concerns tree graphs. Here, we recall the representation of Kemeny's constant of general graphs with no loops given in \cite{kirk:ela}, extend it to the case of graphs with loops, and provide a specific expression for trees containing loops.

In Section \ref{sec:schur}, we briefly survey the definition and properties of the edge centrality measure given in \cite{abcmp} and provide a more convenient formulation that avoids the introduction of a regularization parameter, together with a physical interpretation given in terms of stochastic (Schur) complements. In particular,
we prove that the random walks on the graphs $\hat{\mathc G}_1$ and $\hat{\mathc G}_2$ coincide with the random walk on $\mathc G$ observed in the set of vertices of the graphs  $\hat{\mathc G}_1$ and $\hat{\mathc G}_2$, respectively.
In the same section, we give lower and upper bounds on $c(e)$.

Section \ref{sec:barbell} deals with Kemeny's constant of one-path graphs, i.e., graphs whose adjacency matrix is tridiagonal. This family of graphs describes birth-death stochastic processes. Here, relying on the results of 
 \cite{kirk:ela} and \cite{whitt},
we provide an explicit expression for  the centrality of the edges in a one-path graph  that coincides with the finite truncation of an analogous expression valid for infinite domains given in \cite{bini2018kemeny}. The case of more general trees say, formed by three branches, is considered in Section \ref{sec:2paths}, relying on the extension of the results of \cite{kirk:ela} given in Section \ref{sec:tree} and providing the explicit expression of the adjacency matrix and the distance matrix of this class of graphs. The case of trees formed by $m>3$ branches is outlined in Section \ref{sec:cross}.

Section \ref{sec:exp} reports the results of some numerical experiments while 
Section \ref{sec:conc} summarizes the results and contains final comments.

\section{Preliminaries}\label{sec:prel}
Let $\mathc G=(V,E)$ be an \emph{undirected graph} formed by the set of vertices $V=\{1,2,\ldots,n\}$ and by the set of edges $E\subset V\times V$. The graph $\mathc G$ is fully determined by its
\emph{adjacency matrix} $A=(a_{i,j})$ where $a_{i,j}=1$ if $\{i,j\}\in E$, 
while $a_{i,j}=0$ elsewhere. Since the graph is undirected, the matrix $A$ is symmetric.
We may consider a weighted graph by assigning a weight to each edge so that the adjacency matrix takes values in the  set $\mathbb R^+$ of nonnegative real numbers instead of $\{0,1\}$.

We denote by $\be_i$ the $i$th unit vector of suitable length and by $\ones_n\in\mathbb R^n$  the vector with components equal to 1. If the size is clear from the context, we write $\ones$ instead of $\ones_n$. 

Define $\bm{d}=(d_i)_i=A\ones$. 
If the graph does not have isolated vertices, then $d_i> 0$ for all the values of $i$. 
In this case, the diagonal matrix $D$ whose diagonal entries are $d_1,\ldots,d_n$ is invertible, so that we may define $P=D^{-1}A$. The matrix $P$ has nonnegative entries and satisfies the condition $P\ones=\ones$ so that it is stochastic and defines a Markov chain modeling a random walk in the graph. We refer to $P$ as the \emph{stochastic matrix associated with the graph} $\mathc G$.
The matrix $P$ can be interpreted as the adjacency matrix of a weighted directed graph.
We define the \emph{loop-free directed graph of} $P$, as the weighted directed graph associated with $P$ where the loops are removed, i.e., the graph with adjacency matrix obtained by zeroing the diagonal entries of $P$.

Given a stochastic irreducible matrix $P$, there exists a unique vector $\pig$ such that $\pig^\top P=\pig^\top$, $\pig^\top\ones=1$. Moreover, $\pig$ has strictly positive entries and is the steady state distribution of the Markov chain having $P$ as the transition matrix (see \cite{Seneta:Markov}).

\subsection{Kemeny's constant}\label{sec:kem}
 Kemeny's constant $\kappa(P)$ of the Markov chain associated with an irreducible stochastic matrix $P$ is interpreted in terms of the average expected time to reach a state, randomly chosen according to the probability distribution $\pig$ \cite{kemeny}. It is well known (\cite{kemeny}, \cite{wang2017kemeny})
that
\begin{equation}\label{eq:keminv}
\kappa(P)=\trace((I-P+\ones \bv^\top)^{-1})-1,
\end{equation}
where $\bv$ is any vector such that $\bv^\top\ones=1$. 
The above formula implies that
\begin{equation}\label{eq:kem}
\kappa(P)=\sum_{i=1}^{n-1}\frac1{1-\lambda_i},
\end{equation}
where $\lambda_1,\ldots,\lambda_{n-1},\lambda_n=1$ are the eigenvalues of $P$. We recall that from the Perron-Frobenius theorem \cite{bp:book}, since $P$ is an irreducible stochastic matrix, then $\lambda_n=1$ is the only eigenvalue equal to 1 so that \eqref{eq:kem} is well defined.

Given an undirected connected graph $\mathc G$ with adjacency matrix $A$, we define $\kappa(\mathc G):=\kappa(P)$ where $P=D^{-1}A$ and $D=\diag(d_1,\ldots,d_n)$, $\bm{d}=A\ones$.
The matrix $P$ is similar to the  matrix $\widetilde P=D^{-1/2}A D^{-1/2}$, which is symmetric since $A$ is symmetric. Therefore the eigenvalues of $P$ are real and can be ordered as $-1\le \lambda_1\le \cdots \le \lambda_{n-1}<\lambda_n=1 $.

In the case where $\mathc{G}$ is formed by $m$ disjoint connected components, then $A$ is an $m\times m$ block diagonal matrix, with irreducible diagonal blocks
$A_1,\ldots,A_m$. Consequently,  $P$ is also an $m\times m$ block diagonal matrix with stochastic diagonal blocks $P_1,\ldots,P_m$. Thus, $P$ has exactly $m$ eigenvalues equal to 1 out of $n$ eigenvalues. Therefore,  Kemeny's constant for a disconnected graph is not defined, since in \eqref{eq:kem} we have division by zero in at least one additive term, or in other words the matrix in equation \eqref{eq:keminv} is not invertible.
To deal with this case, by following \cite{abcmp}, we fix a parameter $r\in(0,1)$ and introduce the {\em  regularized Kemeny's constant}
\begin{equation}\label{eq:rks}
\kappa_r(P)=\sum_{i=1}^{n-1}\frac1{1-r\lambda_i}=\trace((I-r(P-\ones \bv^\top))^{-1})-1,
\end{equation}
where $\bv$ is any vector such that $\bv^\top\ones=1$.
Observe that if $\mathc G$ is connected, then $\lim_{r\to 1}\kappa_r(P)=\kappa(P)$, while  if $\mathc G$ is not connected then $\lim_{r\to 1}\kappa(\mathc G)=\infty$. Observe also that if $\mathc G$ is disconnected with $m$ connected components then
\begin{equation}\label{eq:kr}
\kappa_r(P)=\frac{m-1}{1-r}+\sum_{\lambda_j\ne 1}\frac 1{1-r\lambda_j},
\end{equation}
where
\[
 \sum_{\lambda_j\ne 1}\frac 1{1-r\lambda_j}=\sum_{\ell=1}^m\kappa_r(P_\ell).
\]
That is, the regularized Kemeny's constant of a disconnected graph is given by a term that tends to infinity as $r\to 1$ plus a term that tends to the sum of  Kemeny's constants for the connected disjoint subgraphs.

\section{Kemeny's constant of tree graphs with loops}\label{sec:tree}
Recall that a \emph{tree graph} is a graph where for any pair of vertices $i,j$, there is precisely one path  from $i$ to $j$. Given the graph $\mathc G=(V,E)$, we say that the tree graph $\mathc G'=(V',E')$ is
    a \emph{spanning tree of} $\mathc G$ if $V'=V$ and $E'\subset E$. The set $\{\mathc G_i=(V_i,E_i),~~i=1,\ldots,p\}$ is a \emph{spanning forest of} $\mathc G$ if $\mathc G_i$ are tree graphs,  $\cup_i V_i=V$, $V_i\cap V_j=\emptyset$ for $i\ne j$ and $E_i\subset E$.

Suppose that $T$ is a directed tree (i.e. a directed graph formed from a tree by assigning an orientation to each edge). We say that a vertex $v$ of $T$  is a \emph{sink}  if, for each vertex $u \ne v,$ there is a directed path in $T$ from $u$ to $v$. 
For a graph $G$, with adjacency matrix $A,$ we let $F_1$ denote the set of spanning directed trees in the loop-free directed graph of $P=D^{-1}A$,  where $D=\diag(\bd)$ and $\bd=A \ones$, each of which has a sink. We also let $F_2$ denote the set of spanning directed forests in the loop-free directed graph of $G$ such that each forest consists of two directed trees, each of which has a sink. For any spanning directed subgraph $S$ of the loop-free directed graph of $P$, let ${\rm{wt}}(S)$ denote {\emph{weight of}} $S$, i.e. the product of the entries in $P$ that correspond to the arcs in $S$. From Theorem 2.3 of \cite{kirk:ela}, Kemeny's constant for $G$ is given by 
\begin{equation}\label{eq:kcf1f2}
    \kappa(G)= \frac{\sum_{F \in F_2}{\rm{wt}}(F)}{\sum_{T \in F_1}{\rm{wt}}(T)}.
\end{equation}

According to \cite{kirk:ela}, the above formula leads to the following:

\begin{theorem}\label{th:kirk}
Suppose that $\mathc G$ is a connected, undirected graph on $n$ vertices, without loops,
with degree sequence $d_1 , \ldots, d_n$. For each $j, k = 1, \ldots, n$ with $j\ne k$, let $\sigma_{j,k}$ denote the
number of spanning forests consisting of two trees, one of which contains vertex $j$ and
the other of which contains vertex $k$; set $\sigma_{j,j} = 0$, $j = 1, \ldots, n$. Let $\tau$ be the number
of spanning trees in $\mathc G$, $m$ denote the number of edges in $\mathc G$, and let $\Sigma$ be the matrix given by $\Sigma= (\sigma_{ j,k} )$, $ j,k=1,\ldots,n$.  
Then
\[
\kappa(\mathc G ) =\frac{\bd^\top \Sigma \bd}
{4m\tau}.
\]
\end{theorem}

In particular, for a tree graph, $m=n-1$, $\tau=1$ and the above result can be rephrased in a simpler form that involves the matrix $\Delta=(\delta_{i,j})$ of distances, where $\delta_{i,j}$ is the distance of vertex $i$ from vertex $j$.

\begin{corollary}
  If $\mathc G$ is a tree with $n$ vertices and no loops, then
			\[
  \kappa(\mathc G) =	\frac{\bd^\top \Delta \bd }{4(n-1)}, 
			\]
  where $\Delta=(\delta_{i,j})$ is the distance matrix for $\mathc G$.
\end{corollary}

\begin{Remark}\label{rem:bapat}
According to Lemma 8.7 of \cite{Bapat:graphs}, if $\bd$ and $\Delta$ are, respectively, the degree vector and distance matrix of a tree on $n$ vertices, then $\bd^\top \Delta = 2\ones^\top \Delta-(n-1)\ones^\top.$ It follows readily that $\bd^\top \Delta \bd = 4\ones^\top \Delta \ones -2(n-1)(2n-1).$ We note in passing  that the related quantity $\frac{1}{2}\ones^\top \Delta \ones$ is known as the Wiener index, introduced in \cite{Wienerindex}. 
\end{Remark}

Now we prove the following generalization of Theorem \ref{th:kirk} for tree graphs having a loop. 

\begin{theorem} 
Let $\mathc G$ be a  graph satisfying the hypotheses of Theorem \ref{th:kirk}. 
Now select an index $k=1, \ldots, n,$  let $w>0$, and construct the weighted graph $\mathc G(k,w)$ from $\mathc G$ by adding a loop of weight $w$ at vertex $k$. Then the value of Kemeny's constant for the random walk on $\mathc G(k,w)$ is given by 
\begin{equation}\label{eq:kemw} 
\kappa(\mathc G(k,w))=\frac{\bd^\top \Sigma \bd + 2w\bd^\top \Sigma \be_k}{2 \tau(2m+w)}.	
\end{equation}
\end{theorem}
\begin{proof}
		Let $P$ be the transition matrix for the random walk on $\mathc G(k,w)$. Observe that for any index $i=1, \ldots, n$ such that $i \ne k,$ the nonzero entries in the $i$--th row of $P$ are all equal to $\frac{1}{d_i}$, while the nonzero off-diagonal entries in the $k$--th row of $P$ are all equal to $\frac{1}{d_k+w}.$ 
		
		For each $i=1, \ldots , n, $ let $\mathcal{T}_i$ be the collection of spanning directed trees in $\mathc G(k,w)$ having vertex $i$ as a sink. For  each $i, j=1, \ldots , n, $ let $\mathcal{F}_{ij}$ be the collection of spanning directed forests in $\mathc G(k,w)$ consisting of trees such that  vertex $i$ as a sink in one tree and vertex $j$ as a sink in the other. (Note that $\mathcal{T}_i$ and  $\mathcal{F}_{ij}$ coincide with the corresponding collections in $\mathc G$.)  
		
		From \eqref{eq:kcf1f2} the value of Kemeny's constant is equal to $$\frac{\sum_{i,j=1, \ldots, n} \sum_{F \in  \mathcal{F}_{ij}}   {\rm{wt}}(F)}
		{\sum_{i=1}^n\sum_{T\in \mathcal{T}_i} {\rm{wt}}(T)} .$$
		
		Fix a spanning tree $\mathc S$ of $\mathc G$. For each $i=1, \ldots, n,$ there is precisely one orientation of the edges of $\mathc S$ so that vertex $i$ is a sink. Denote the corresponding directed tree by $\mathc T$. If $i =k,$ then ${\rm{wt}}(\mathc T) = \frac{d_k}{\prod_{j=1}^n d_j}.$ On the other hand if $i \ne k,$ then ${\rm{wt}}(\mathc T) = \frac{d_i}{(d_k+w)\prod_{j \ne k} d_j}.$ Consequently  each spanning tree of $\mathc G$ contributes $\frac{d_k}{\prod_{j=1}^n d_j} + \sum_{i \ne k} \frac{d_i}{(d_k+w)\prod_{j \ne k} d_j}$ to $	{\sum_{i=1}^n\sum_{T\in \mathcal{T}_i} {\rm{wt}}(T)}.$ Observing that 
		\begin{eqnarray*}
	&&	\frac{d_k}{\prod_{j=1}^n d_j} + \sum_{i \ne k} \frac{d_i}{(d_k+w)\prod_{j \ne k} d_j}		
		 = \\
	&&	  \frac{1}{(d_k+w)\prod_{j \ne k} d_j}(\sum_{i \ne k}d_i + (d_k+w) ) = \\
	&&	\frac{1}{(d_k+w)\prod_{j \ne k} d_j}(2m +w),
			\end{eqnarray*}
		it now follows that 
		\[
		{\sum_{i=1}^n\sum_{T\in \mathcal{T}_i} {\rm{wt}}(T)}
		= 	\frac{1}{(d_k+w)\prod_{j \ne k} d_j}(2m +w) \tau. 
		\]
		
		Next, fix distinct indices $p,q$ and a directed forest $F \in  \mathcal{F}_{pq}.$ If $p, q \ne k,$ then  ${\rm{wt}}(F) = \frac{d_pd_q}{(d_k+w)\prod_{j \ne k} d_j}.$ On the other hand if one of $p, q$ is equal to $k,$ say $p=k,$ then  ${\rm{wt}}(F) = \frac{(d_k+w)d_q}{(d_k+w)\prod_{j \ne k} d_j} =\frac{d_kd_q}{(d_k+w)\prod_{j \ne k} d_j} + \frac{wd_q}{(d_k+w)\prod_{j \ne k} d_j} .$ Hence we find that 
		$$\sum_{p,q=1, \ldots, n} \sum_{F \in  \mathcal{F}_{pq}}   {\rm{wt}}(F)
 = \sum_{p, q=1, \ldots, n} \frac{d_kd_q\sigma_{pq}}{(d_k+w)\prod_{j \ne k d_j}}
		+ \sum_{q=1, \ldots, n} \frac{wd_q \sigma_{qk}}{(d_k+w)\prod_{j \ne k} d_j} . $$ 	Since 
		$\sum_{p, q=1, \ldots, n} {d_kd_q\sigma_{pq}} = \frac{1}{2}\bd^\top \Sigma \bd$ and 
	 $\sum_{q=1, \ldots, n} d_q\sigma_{qk} = \bd^\top \Sigma \be_k,$ \eqref{eq:kemw} follows readily. 
\end{proof}

\begin{corollary}\label{cor:2}
In the case that our graph is a tree with $n$ vertices, \eqref{eq:kemw} reduces to 
			\[
			\kappa(\mathc G(k,w))=\frac{\bd^\top \Delta \bd + 2 w\bd^\top \Delta \be_k}{2 (2n-2+w)}, 
			\]
			where $\Delta$ is the distance matrix for $\mathc G$.
	\end{corollary}

\begin{Remark}\label{rem:alt}
Here we maintain the notation of Corollary \ref{cor:2}. We have already observed in Remark \ref{rem:bapat} that $\bd^\top \Delta \bd$ can be written in terms of $\ones^\top \Delta \ones$; further, since $\bd^\top \Delta \be_k = 2 \ones^\top \Delta\be_k-(n-1),$ we arrive at the following alternate expression for $\kappa(\mathc G(k,w))$, which is sometimes more convenient to work with: 
 $$\kappa(\mathc G(k,w)) = \frac{2\ones^\top \Delta \ones + 2w  \ones^\top \Delta \be_k -(n-1)(2n-1+w)}{2n-2+w}. 
 $$    
\end{Remark}

\section{Centrality of cut-edges}\label{sec:schur}
In \cite{abcmp} a measure of the centrality or importance of an edge $e$ in a connected graph $\mathc G=(V,E)$ is defined in terms of the variation of  Kemeny's constant. 
This definition has a different expression, depending on whether or not $e$ is a cut-edge. 
In the case that $e$ is a cut-edge, the adjacency matrix $A$ of  $(V, E\setminus \{e\})$ is a $2\times 2$ block diagonal matrix, i.e., $A=\hbox{diag}(A_1,A_2)$, so that $(V, E\setminus \{e\})$ is formed by two disjoint connected graphs  $\mathc G_1$ and $\mathc G_2$, having respective  adjacency matrices $A_1$ and $A_2$.

In this section, we briefly survey the definition and  properties of the edge centrality measure given in \cite{abcmp}. 

Let $\widehat{\mathc{G}}$ be the graph obtained by removing the edge $e=(i,j)$ and by adding a loop at vertices $i$ and $j$, with the same weight as the edge $e$. The adjacency matrix of $\widehat{\mathc G}$ is given by 
\begin{equation}\label{eq:rk1}
\widehat A=A+a_{ij}(\be_i-\be_j)(\be_i-\be_j)^\top.
\end{equation}

In \cite{abcmp}, if $e$ is not a cut-edge, the centrality measure of $e$ is defined as $c(e)=\kappa(\hat{\mathc G})-\kappa(\mathc G)$, that is the variation of Kemeny's constant for  the graph resulting from the removal of the edge $e$ and  the addition of the two loops.

Instead, if $e$ is a cut-edge, $\hat{\mathc G}$ is disconnected so that $c(e)=\infty$. 
For this reason, the concept of the regularized centrality measure has been introduced in \cite{abcmp}. According to this definition, given  the regularization parameter $r>0$, the regularized centrality score of $e$ is
\begin{equation}\label{eq:cer}
c_r(e)=\frac1{1-r}-\left( \kappa_r(\hat{\mathc G})-\kappa_r(\mathc G) \right),
\end{equation}
where $\kappa_r(\cdot)$ is the regularized Kemeny's constant defined in \eqref{eq:rks}.

The following proposition gives an expression for the limit of $c_r(e)$ as $r\to 1$:
\begin{proposition}
Let $e=(i,j)\in E$ be a cut edge  and let $\widehat{\mathc G}$ be the graph obtained by removing edge $e$ and by adding two loops at vertices $i$ and $j$ with weights $a_{ij}$. Let $\widehat{\mathc G}_1$ and $\widehat{\mathc G}_2$  be the two connected components of $\widehat{\mathc G}$.
Then
\[
\lim_{r\to 1}c_r(e)=\kappa(\mathc G)-\kappa(\hat{\mathc G}_1)-\kappa(\hat{\mathc G}_2)>0.
\]
\end{proposition}

\begin{proof}
The proof follows from \eqref{eq:kr}, applied with $m=1$, and from the continuity of eigenvalues.
\end{proof}
Summarizing this analysis, we introduce the definition of the centrality measure of an edge $e$  as
\begin{equation}\label{eq:centr}
    c(e)=\left\{
    \begin{array}{ll}
        \kappa(\hat{\mathc{G}})-\kappa(\mathc{G}) & \text{if } e \text{ is not a cut-edge}\\
          \kappa(\mathc{G})-\kappa(\hat{\mathc G}_1)-\kappa(\hat{\mathc G}_2)  & \text{if } e \text{ is a cut-edge}.
    \end{array}
    \right.
\end{equation}
It is worth pointing out that, with respect to the definition given in \cite{abcmp}, this formulation does not require regularization.

We prove a result that provides a physical interpretation of the graphs $\hat{\mathc G}_1$ and $\hat{\mathc G}_2$ obtained by removing a cut-edge $e$ and adding  loops at its end vertices. This result, expressed in terms of stochastic (Schur) complements,  adds some more information about the properties of  Kemeny's constant of stochastic complements given in \cite{bdkm}. In the following, given an adjacency matrix $A$, we denote by $D_A$ the diagonal matrix whose diagonal entries are the components of the vector $A\ones$.

\begin{theorem}\label{th:schur}
Let $\mathc G$ be a connected graph with adjacency matrix $A=(a_{ij})$, $e=(i,j)$ a cut-edge, and $\hat{\mathc G}_1$, $\hat{\mathc G}_2$ the two connected graphs obtained from $\mathc G$ by removing the edge $e$ and adding two loops at the vertices of $e$ with weight $a_{ij}$. Denote by $\widehat A_1$ and $\widehat A_2$ the adjacency matrices of the graphs  $\hat{\mathc G}_1$, and $\hat{\mathc G}_2$, respectively. Denote by $P=D_A^{-1}A$, $\widehat P_1=D_{\widehat A_1}^{-1}\widehat A_1$ and $\widehat P_2=D_{\widehat A_2}^{-1}\widehat A_2$ the three stochastic matrices associated with a random walk on $\mathc G$, $\hat{\mathc G}_1$, and $\hat{\mathc G}_2$, respectively. Then $\widehat P_1$ and $\widehat P_2$ are the stochastic complements of $P$, that is, $\widehat P_1=P_{11}+P_{12}(I-P_{22})^{-1}P_{21}$,
$\widehat P_2=P_{22}+P_{21}(I-P_{11})^{-1}P_{12}$, where $I$ denotes the identity matrix of suitable size, 
\[
P=\begin{bmatrix}
P_{11}&P_{12}\\ P_{21}&P_{22}
\end{bmatrix}
\]
and the vertices are numbered so that the vertices of $\widehat {\mathc G}_1$ come first.
\end{theorem}
\begin{proof}
Let $m$ denote  the size of $P_{11}$, i.e., the number of vertices in $\widehat{\mathc G}_1$.
Since $e=(i,j)$ is a cut-edge, then $1\le i\le m$, and $j>m$, moreover all the entries of $P_{12}$ are zero except the entry in position $(i,j-m)$ which is $p_{ij}$, so that we may write $P_{12}=p_{ij}\be_i\be_{j-m}^\top$ and $P_{21}=p_{ji}\be_{j-m}\be_i^\top$. Now consider the stochastic complement $P_1$. We find that
\[
\widehat P_1=P_{11}+p_{ij}p_{ji}\be_i\be_{j-m}^\top(I-P_{22})^{-1}\be_{j-m}\be_i^\top=P_{11}+\gamma \be_i\be_i^\top,
\]
where $\gamma= p_{ij}p_{ji} \be_{j-m}^\top(I-P_{22})^{-1}\be_{j-m}
$.
Since $\widehat P_1$ is stochastic, then $\gamma=1/d_i$ where $d_i$ is the $i$th component of the vector $\bd=A\ones$. Thus, $\widehat P_1$ differs from $P_{11}$ for the entry in position $(i,i)$ that is equal to $1/d_i$. Therefore, the matrix $\widehat P_1$ coincides with the stochastic matrix associated with the graph $\widehat{\mathc G}_1$ where the loop at vertex $i$ is given exactly by the entry $1/d_i$ in position $(i,i)$ (compare with \eqref{eq:rk1}).
We proceed similarly  for $\widehat P_2$.
\end{proof}

According to the above theorem, the random walk on the graph $\hat{\mathc G}_i$ coincides with the random walk on the whole graph $\mathc G$ observed in the vertices of the subgraph $\widehat{\mathc G}_i$, for $i=1,2$. This fact provides a physical interpretation to the definition of $c(e)$ given in \eqref{eq:centr}  for a cut-edge $e$.

A lower and upper bound on $c(e)$  can be obtained by the Cauchy interlacing theorem:

\begin{theorem}\label{thm:bounds}
Let $A$ be the adjacency matrix associated with the 
connected graph $\mathc G$.
Let $P=D_A^{-1}A$ be the associated stochastic matrix. Let $e$ be a cut-edge in $\mathc G$. Then
\[
\frac1{1-\lambda_1}\le c(e)\le \frac1{1-\lambda_{n-1}},
\]
where $-1\le\lambda_1\le\lambda_2\le\cdots\le \lambda_{n-1}<\lambda_n=1$ are the eigenvalues of $P$.
\end{theorem}
\begin{proof}
        Let $e=(p,q)$ and set $\bv=\be_p-\be_q$. According to \eqref{eq:rk1}, the adjacency matrix $\widehat A$ associated with $\widehat {\mathc G}$ is $\widehat A=A+a_{pq}\bv\bv^\top=\hbox{diag}(\widehat A_1,\widehat A_2)$,
        where $\widehat A_1$ and $\widehat A_2$ are the adjacency matrices associated with $\widehat{\mathc G}_1$ and $\widehat{\mathc G}_2$. Therefore for the vectors
        $\bd=A\ones$ and $\widehat \bd=\widehat A\ones$ we have
        $\bd=\widehat \bd$.  This implies that $D_A=D_{\widehat A}$ and 
        \[
        P=D_A^{-1}A,\quad \widehat P=\hbox{diag}(\widehat P_1,\widehat P_2)=D_A^{-1}(A+ a_{pq}
        \bv\bv^\top).
        \]
        The above equation can be symmetrized as follows
        \[\begin{split}
        &D_A^{\frac12}PD_A^{-\frac12}=D_A^{-\frac12}AD_A^{-\frac12},\\ 
        &D_A^{\frac12}\widehat PD_A^{-\frac12}=
        D_A^{-\frac12}\widehat AD_A^{-\frac12}=
        D_A^{-\frac12}\hbox{diag}(\widehat P_1,\widehat P_2)D_A^{-\frac12}=D_A^{-\frac12}A D_A^{-\frac12}+a_{pq}
        \bw\bw^\top,
        \end{split}
        \]
        for $\bw=D_A^{-\frac12}\bv$.
        Therefore, by applying the Cauchy interlacing theorem \cite{hj:book} to the symmetric matrices
        $D_A^{-\frac12}AD_A^{-\frac12}$ and $D_A^{-\frac12}\widehat AD_A^{-\frac12}$ whose difference is a rank 1 matrix,
        we find that 
        \begin{equation}\label{eq:ineq}
        \lambda_i\le\mu_i\le\lambda_{i+1},
        \end{equation}
        where $\mu_i$ are the eigenvalues of $\widehat P$ sorted in non decreasing order. Moreover, since the set of values $\mu_i$ is the union of the spectra of $\widehat P_1$ and $\widehat P_2$ that are both stochastic, we have $\mu_{n-1}=\mu_n=1$.
        Finally recall that $\kappa(\mathc G)=\sum_{i=1}^{n-1}\frac1{1-\lambda_i}$ and $\kappa(\widehat{\mathc G}_1)+\kappa(\widehat{\mathc G}_2)=\sum_{i=1}^{n-2}\frac1{1-\mu_i}$, so that from \eqref{eq:ineq} we deduce that
        \[
        \sum_{i=1}^{n-2}\frac1{1-\lambda_{i}}\le
        \kappa(\widehat {\mathc G}_1)+\kappa(\widehat {\mathc G_2})\le\sum_{i=1}^{n-2}\frac1{1-\lambda_{i+1}},
        \]
        that is
        \[
        \kappa(\mathc G)-\frac1{1-\lambda_{n-1}}\le
        \kappa(\widehat {\mathc G}_1)+\kappa(\widehat {\mathc G_2})\le\kappa(\mathc G)-\frac1{1-\lambda_{1}},
        \]
        whence
        \[
        \frac1{1-\lambda_1}\le \kappa(\mathc G)-\kappa(\widehat{\mathc G}_1)-\kappa(\widehat{\mathc G}_2)\le \frac1{1-\lambda_{n-1}}
        \]
        which completes the proof since $c(e)=\kappa(\mathc G)-\kappa(\widehat{\mathc G}_1)-\kappa(\widehat{\mathc G}_2)$.
\end{proof}

Our next two examples illustrate the fact that the bounds on $c(e)$ in Theorem~\ref{thm:bounds} can be quite accurate.

\begin{example}\label{ex1}
Consider the star on $n$ vertices, $K_{1,n-1},$ which has a vertex of degree $n-1$ adjacent to $n-1$ vertices of degree $1$, see Figure \ref{fig:ex}. By assuming that the centre of the star is vertex $n$, the transition matrix for the corresponding random walk is $\left[\begin{smallmatrix}
	0 &\ones_{n-1}\\ \frac{1}{n-1}\ones_{n-1}^\top & 0\end{smallmatrix}\right] $, which has eigenvalues $-1, 1, 0,$ the latter of multiplicity $n-2$. Hence $\lambda_{n-1}=0$ for this graph, and, from \eqref{eq:kem}, it is readily determined that $\kappa(K_{1,n-1}) = n-\frac{3}{2}.$ Now let $e$ be one of the edges of $K_{1,n-1}$. Assume, without loss of generality, that $e$ is the edge $\{1,n\}$. The transition matrix that arises by deleting $e$ and adding a loop at the endpoints of $e$ can be written as 
    \[\begin{bmatrix}
    1&0&0\\
    0&0&\ones_{n-2}\\
    0&\frac1{n-1}\ones_{n-2}^\top&\frac1{n-1}
\end{bmatrix}.
\]
Kemeny's constant for the trailing principal submatrix of order $n-1$ is equal to $n-3+\frac{n-1}{2n-3},$ while Kemeny's constant for the leading principal submatrix of order $1$ is equal to $0$. Consequently, we have $c(e) = n-\frac{3}{2} - (n-3+\frac{n-1}{2n-3}) - 0 = \frac{4n-7}{4n-6}.$ The upper bound of Theorem \ref{thm:bounds} is $\frac{1}{1-\lambda_{n-1}} =1,$ and thus we see that for large $n,$ $c(e)$ is close to that upper bound. 

\end{example}

\begin{figure}
\begin{center}
    \begin{tabular}{cc}
\begin{tikzpicture}
[
      mycircle/.style={
         circle,
         draw=black,
         fill=none,
         inner sep=0pt,
         minimum size=15pt,
         font=\small},
      myarrow/.style={},  
     node distance=0.7cm and 1.3cm
      ]
    \node[mycircle] at (360:0mm) (center) {};
    \foreach \n in {1,...,7}{
        \node[mycircle] at ({\n*360/7}:1.5cm) (n\n) {\n};
        \draw (center)--(n\n);
    }
     \node at (center)  {8} ;
\end{tikzpicture}
    &\quad\quad\quad
\begin{tikzpicture}
  \graph[nodes={draw, circle}, clique, n=5, clockwise, radius=1.5cm]
  {
    1/"1", 2/"5", 3/"4", 4/"3", 5/"2"
  };
  \node [draw, circle, right = 1.2cm of 2] (6) {6};
  \draw (2) edge node[circle] {} (6);
\end{tikzpicture}

    \end{tabular}\caption{On the left, the star graph of Example~\ref{ex1} for $n=8$. On the right, the graph of Example~\ref{ex2}  with $n=6$.}\label{fig:ex}
    \end{center}
    \end{figure}
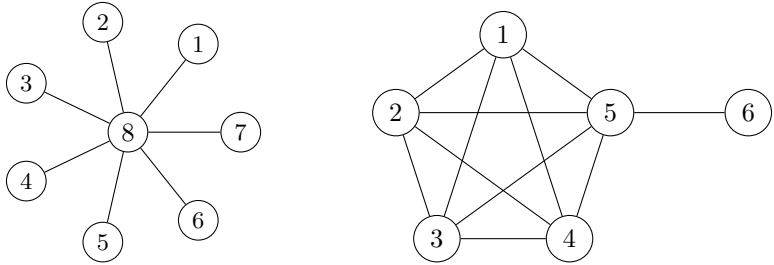

\begin{example}\label{ex2}
Let $K_{n-1}$ be the complete graph formed by $n-1$ vertices such that for any $i,j\in\{1,\ldots,n-1\}$, $i\ne j$ there is an edge connecting $i$ to $j$.
Let $G$ be the graph formed from $K_{n-1}$ by adding the pendent edge $e=\{n-1,n\}$, see Figure \ref{fig:ex}. The corresponding transition matrix is $\left[\begin{smallmatrix}
	\frac{1}{n-2}(J-I)&\frac{1}{n-2}\ones_{n-2} &0\\
\frac{1}{n-1}\ones_{n-2}^\top &0 &\frac{1}{n-1}\\
0 &1&0
\end{smallmatrix}\right]$, where $J$ is the all-ones matrix of the appropriate order. The eigenvalues of the transition matrix are: $1,-\frac{1}{n-2}$ (with multiplicity $n-3$), and $\frac{1}{2}\left(-\frac{1}{n-2} \pm \sqrt{\frac{4n^2-19n+23}{(n-1)(n-2)^2}}\right).$ We thus find that for $n \ge 3, \lambda_1 =  -\frac{1}{2}\left(\frac{1}{n-2} +  \sqrt{\frac{4n^2-19n+23}{(n-1)(n-2)^2}}\right).$ From these eigenvalues and \eqref{eq:kem}, we deduce that $\kappa(G) = \frac{(n-3)(n-2)}{(n-1)}+ \frac{2n^2-5n+3}{n^2-3n+4}.$

Deleting the pendent edge $e$  from $G$ and adding a loop at the endpoints of $e$  yields the following transition matrix: \[
\left[\begin{matrix}
	\frac{1}{n-2}(J-I) & \frac{1}{n-2}\ones_{n-2} &0\\\frac{1}{n-1}\ones_{n-2}^\top &\frac{1}{n-1} &0\\0&0&1
\end{matrix}\right].
\]
The eigenvalues for the leading principal submatrix of order $n-1$ are $1, -\frac{1}{n-2}$ (with multiplicity $n-3$) and $-\frac{1}{(n-1)(n-2)}.$ We then find that for that leading principal submatrix, Kemeny's constant is equal to  $\frac{(n-3)(n-2)}{(n-1)} + \frac{(n-1)(n-2)}{n^2-3n+3}.$ It now follows that $c(e) = 
 \frac{(n-3)(n-2)}{(n-1)}+ \frac{2n^2-5n+3}{n^2-3n+4} - \left(\frac{(n-3)(n-2)}{(n-1)} + \frac{(n-1)(n-2)}{n^2-3n+3}\right) - 0 = \frac{2n^2-5n+3}{n^2-3n+4} - \frac{(n-1)(n-2)}{n^2-3n+3}.$ For large $n$ we see that both $c(e)$  and $\frac{1}{1-\lambda_1}$ are near to $1$, so that $c(e)$ is close to the lower bound of Theorem \ref{thm:bounds} for such $n$. 
\end{example}

\section{Centrality of edges in one-path graphs}\label{sec:barbell}
In this section and  the next, 
we prove that for certain classes of graphs, the quantity $c(e)$ behaves in the expected way in the sense that the importance of the edge $e$ connecting two disjoint connected components $\mathc G_1$ and $\mathc G_2$ of the graph is larger the larger is the minimum of the cardinalities
of $\mathc G_1$ and $\mathc G_2$.

A random walk on a one-path graph can be seen as a particular birth-death process, i.e., a Markov chain with tridiagonal transition matrix
\begin{equation}\label{eq:Pqbd}
    P=\begin{bmatrix}
 \theta_0& \lambda_0&\\
\mu_1 & \theta_1 & \lambda_1\\
&\ddots & \ddots & \ddots \\
&&\mu_{n-2} & \theta_{n-2} &\lambda_{n-2}\\
&&&\mu_{n-1}&\theta_{n-1}
\end{bmatrix},
\end{equation}
where $\theta_0=1-\lambda_0\ge 0$, $\theta_i=1-\lambda_i-\mu_i\ge 0$ for $i=1,\ldots,n-2$, $\theta_{n-1}=1-\lambda_{n-1}$, and $0<\lambda_i<1$, $0<\mu_i<1$.

Kemeny's constant for $P$ can be expressed explicitly in different forms. 
Let us first introduce the quantities
\[
\begin{split}
&t(k)=\prod_{\ell=0}^{k-1}\lambda_\ell \prod_{\ell = k+1}^{n-1}\mu_\ell,~~,k=0, \ldots, n-1,\\
    & 
    w(j,k,p)=\prod_{\ell=0}^{j-1}\lambda_\ell\prod_{\ell = j+1}^{p}\mu_\ell \prod_{\ell=p+1}^{k-1}\lambda_\ell \prod_{\ell = k+1}^{n-1}\mu_\ell,~~0\le j < k \le n-1,j\le p \le k-1,
\end{split}
\]
where we interpret an empty product as $1$. 

The following theorem provides two equivalent expressions for $\kappa(P)$, one given in terms of the parameters $t(k)$ and $w(j,k,p)$, and one expressed in terms of the stationary vector $\pig$ of $P$.

\begin{theorem}\label{th:bd}
Let $P$ be the irreducible stochastic $n\times n$ matrix defined in \eqref{eq:Pqbd}. Then
\begin{equation}\label{eq:kqbd1}
\kappa(P)=\frac{\sum_{0\le j<k\le n-1}\sum_{p=j}^{k-1}w(j,k,p)}{\sum_{k=0}^{n-1}t(k)}.
\end{equation}
Equivalently,
\begin{equation}\label{eq:kqbd}
\kappa(P)=\sum_{k=0}^{n-2} \frac{\sigma_k(1-\sigma_k)}{\lambda_k \pi_k},
\quad \sigma_k=\sum_{j=0}^{k}\pi_j,
\end{equation}
where
 $\pig=(\pi_i)_{i=0,\ldots,n-1}$ is such that  $\pig^\top P=\pig^\top$, $\pig^\top\ones=1$. Moreover, we have
 \begin{equation}\label{eq:pi}
 \pi_j = \pi_0 \frac{\prod_{\ell = 0}^{j-1} \lambda_\ell}{\prod_{\ell=1}^j \mu_\ell},~j=1,\ldots,n-1,\quad \hbox{with} \quad \pi_0 = \left(\sum_{j=0}^{n-1} \frac{\prod_{\ell = 0}^{j-1} \lambda_\ell}{ \prod_{\ell=1}^j \mu_\ell}\right)^{-1}.
 \end{equation}
\end{theorem}

\begin{proof}
To show \eqref{eq:kqbd1} we use formula \eqref{eq:kcf1f2}.
Consider the loop-free directed graph associated with $P$. For each $k=0, \ldots, n-1,$ there is one spanning directed tree having $k$ as a sink, and its arc set is $\{\ell \rightarrow \ell+1: \ell =0, \ldots, k-1 \} \cup \{ \ell \rightarrow \ell-1: \ell=k+1, \ldots, n-1\}.$ Evidently the corresponding weight is equal to $t(k)=\prod_{\ell=0}^{k-1}\lambda_\ell \prod_{\ell = k+1}^{n-1}\mu_\ell.$ For the directed forests in $F_2$, suppose that such a directed forest has vertex $j$ as the sink in one directed tree and vertex $k$ as the sink in the other directed tree; without loss of generality, $j <k.$ For each such directed forest, there is a vertex $p$ with $j \le p < p+1\le k$ so that the corresponding arc set is given by 
$\{\ell \rightarrow \ell+1: \ell =0, \ldots, j-1 \} \cup \{ \ell \rightarrow \ell-1: \ell=j+1, \ldots, p\} \cup \{\ell \rightarrow \ell+1: \ell =p+1, \ldots, k-1 \} \cup \{ \ell \rightarrow \ell-1: \ell=k+1, \ldots, n-1\}.$ 
The weight of such a forest is given by $w(j,k,p)=\prod_{\ell=0}^{j-1}\lambda_\ell\prod_{\ell = j+1}^{p}\mu_\ell \prod_{\ell=p+1}^{k-1}\lambda_\ell \prod_{\ell = k+1}^{n-1}\mu_\ell  .$ From \eqref{eq:kcf1f2} we thus find \eqref{eq:kqbd1}.

The expression \eqref{eq:kqbd} can be derived from equation (30) in \cite{whitt}, which we recall below
\begin{equation}\label{eq:witt}
    \pi_j^{-1}Z_{jj}=\sum_{k=0}^{j-1}\frac{\sigma_k^2}{\lambda_k\pi_k}+\sum_{k=j}^{n-1}\frac{(1-\sigma_k)^2}{\lambda_k\pi_k},
\end{equation}
where
\[
Z=(I -P + \ones\pig^\top)^{-1}-\ones\pig^\top.
\]
In fact, from the above equation and from \eqref{eq:keminv} applied with $\bm v=\pig$, we deduce that $\kappa(P)=\hbox{trace}(Z+\ones\pig^\top)-1=\hbox{trace}(Z)$. Therefore, from \eqref{eq:witt} we may write
\[
\kappa(P)=\hbox{trace}(Z)=\sum_{j=0}^{n-1}\pi_j\sum_{k=0}^{j-1}\frac{\sigma_k^2}{\pi_k\lambda_k}+\sum_{j=0}^{n-1} \pi_j\sum_{k=j}^{n-1}\frac{(1-\sigma_j)^2}{\pi_j\lambda_j}.
\]
Changing the order of the two summations in each summand of the above expression we obtain

\[
\kappa(P)
=\sum_{j=0}^{n-2} \frac{\sigma_j^2}{\pi_j\lambda_j}(1-\sigma_j)+
\sum_{j=0}^{n-2}\frac{(1-\sigma_j)^2}{\pi_j\lambda_j}\sigma_j
=\sum_{j=0}^{n-2}
\frac{\sigma_j(1-\sigma_j)}{\pi_j\lambda_j},
\]
which proves \eqref{eq:kqbd}.
Equation \eqref{eq:pi} follows by a direct check.
The equivalence of \eqref{eq:kqbd} and \eqref{eq:kqbd1} can be obtained by substituting the expression of $\pig$ given by \eqref{eq:pi} into \eqref{eq:kqbd}.
\end{proof}

It is interesting to point out that \eqref{eq:kqbd} formally coincides with the truncation to the finite size of the infinite summation in the expression for $\kappa(P)$ given in \cite[Theorem 5.1]{bini2018kemeny} in the case where $P$ is semi-infinite.

\begin{figure}
    \centering
\begin{tikzpicture}[
      mycircle/.style={
         circle,
         draw=black,
         fill=gray,
         fill opacity = 0.3,
         text opacity=1,
         inner sep=0pt,
         minimum size=5pt,
         font=\small},
      myarrow/.style={},  
     node distance=0.7cm and 1.3cm
      ]
      \node[mycircle] (c1) [label=above:$1$]{};
      \node[mycircle,right=of c1] (c2) [label=above:$2$]{};
      \node[mycircle,right=of c2] (c3) [label=above:$3$]{};
      \node[mycircle,right=of c3] (c4) [label=above:$4$]{};
      \node[mycircle,right=of c4] (c5) [label=above:$5$]{};
      \node[mycircle,right=of c5] (c6) [label=above:$6$]{};

   \draw[myarrow] (c1) -- (c2);
   \draw[myarrow] (c4) -- (c5);
   \draw[myarrow] (c2) -- (c3);
   \draw[myarrow] (c3) -- (c4);
  \draw[myarrow] (c5) -- (c6);
  \draw (c1) edge [loop below, in=-60, out=-120, min distance=1cm]  (c1);
   \draw (c6) edge [loop below, in=-60, out=-120, min distance=1cm]  (c6);
    \end{tikzpicture}

\begin{tikzpicture}[
      mycircle/.style={
         circle,
         draw=black,
         fill=gray,
         fill opacity = 0.3,
         text opacity=1,
         inner sep=0pt,
         minimum size=5pt,
         font=\small},
      myarrow/.style={},  
     node distance=0.7cm and 1.3cm
      ]
      \node[mycircle] (c1) [label=above:$1$]{};
      \node[mycircle,right=of c1] (c2) [label=above:$2$]{};
      \node[mycircle,right=of c2] (c3) [label=above:$3$]{};
      \node[mycircle,right=of c3] (c4) [label=above:$4$]{};
      \node[mycircle,right=of c4] (c5) [label=above:$5$]{};
      \node[mycircle,right=of c5] (c6) [label=above:$6$]{};

   \draw[myarrow] (c1) -- (c2);
   \draw[myarrow] (c4) -- (c5);
   \draw[myarrow] (c2) -- (c3);
  \draw[myarrow] (c5) -- (c6);
  \draw (c1) edge [loop below, in=-60, out=-120, min distance=1cm]  (c1);
   \draw (c3) edge [loop below, in=-60, out=-120, min distance=1cm]  (c3);
    \draw (c4) edge [loop below, in=-60, out=-120, min distance=1cm]  (c4);
   \draw (c6) edge [loop below, in=-60, out=-120, min distance=1cm]  (c6);
    \end{tikzpicture}

    \caption{Class of graphs formed by $n$ vertices. Vertex $i$ is connected by an edge to vertex $i+1$, for $i=1,\ldots,n-1$; vertices 1 and $n$ may have a loop. Removing the edge $(m,m+1)$ with the methodology of \cite{abcmp} yields two disjoint graphs that belong to the same class.}
    \label{fig:1}
\end{figure}
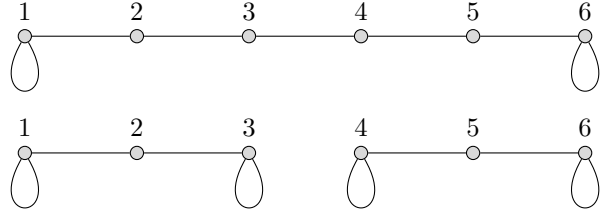

Let us consider the class of graphs depicted in Figure \ref{fig:1}, associated with the $n\times n$ adjacency matrix
\begin{equation}\label{eq:adj}
A_n=A_n(\alpha,\beta)=\begin{bmatrix}
\alpha&1&\\
1&0&1\\
&\ddots & \ddots & \ddots \\
&&1 &0 &1\\
&&&1&\beta
\end{bmatrix}\in\mathbb{R}^{n\times n}.
\end{equation}

We have the following result which generalizes Proposition 3.2 in \cite{faught2021} concerning path graphs.

\begin{theorem}\label{thm:1}
Kemeny's constant $\kappa_n(\alpha,\beta)$ of the graph associated with the adjacency matrix \eqref{eq:adj} is
\[
\kappa_n(\alpha,\beta)=(n-1)\frac{\frac23n^2+n(\alpha+\beta-\frac43)+\alpha\beta-\alpha-\beta+1}{2n+\alpha+\beta-2}.
\]
\end{theorem}
\begin{proof}
We rely on equation \eqref{eq:kqbd1}.
Observe that, for this graph, 
\begin{description}
\item{i)~} $t(0)=\frac{1}{2^{n-2}(\alpha+1)}$, 
\item{ii)} $t(n-1)=\frac{1}{2^{n-2}(\beta+1)}$,
\item{iii)} $t(k)=\frac{1}{2^{n-3}(\alpha+1)(\beta+1)},\quad  k=1, \ldots, n-2$, 
\item{iv)} $w(j,k,p)=\frac{1}{2^{n-4}(\alpha+1)(\beta+1)}, ~~1\le j < k \le n-2,$ 
\item{v)~} $w(0,k,p)=\frac{1}{2^{n-3}(\beta+1)}, ~~1\le k \le n-2,$ 
\item{vi)} $w(j,n-1,p)=\frac{1}{2^{n-3}(\alpha+1)},~~ j=1, \ldots, k-1,$ 
\item{vii)} $w(0,n-1,p)=\frac{1}{2^{n-2}}$.
\end{description}
Combining the above expressions yields  the desired  representation for $\kappa_n(\alpha,\beta)$.
\end{proof}

From the above results we immediately derive the following consequence

\begin{corollary}
The values of  Kemeny's constant $\kappa(\alpha,\beta)$  of the graph associated with $A_n$ for $\alpha,\beta\in\{0,1\}$ are:   
\begin{enumerate} 
\item $\kappa_n(\alpha,1)=\frac13\frac{n(n-1)(2n+3\alpha-1)}{2n+\alpha-1}$,

\item
$\kappa_n(0,0)=\frac13 (n^2-2n+\frac32)$, 

\item
$\kappa_n(0,1)=\kappa_n(1,0)=\frac13 n(n-1)$, 

\item
$\kappa_n(1,1)=\frac13 (n^2-1)$. 

\end{enumerate}
\end{corollary}

It is interesting to point out that the expression of $\kappa_n(0,0)$ coincides with the one given in \cite{kirk:ela,breen19}.
Observe that, by removing the edge connecting vertex $m$ with vertex $m+1$, with $1\le m<n$, in the graph associated with the adjacency matrix $A_n(\alpha,\beta)$ in equation \eqref{eq:adj} by means of a symmetric rank-1 correction, we obtain the block diagonal adjacency matrix 
\[
\widehat A=
\begin{bmatrix}
A_m(\alpha,0)&0\\0& A_{n-m}(0,\beta)
\end{bmatrix}.
\]

\begin{theorem}
The centrality $c(e)$ of the edge $e=(m,m+1)$ in the graph $\mathc G$ associated with the adjacency matrix $A_n$ in \eqref{eq:adj},  is 
\[
c(e)=\kappa_n(\alpha,\beta) -\kappa_m(\alpha,1)-\kappa_{n-m}(\beta,1) .
\]
In particular, for $\alpha=\beta=0$ we have
$c(e)=\frac13(2m(n-m)-n+\frac32)$, while for $\alpha=\beta=1$ we have $c(e)=\frac13(2m(n-m)+1)$.
\end{theorem}
\begin{proof}
It follows from Theorem \ref{thm:1} in view of the definition \eqref{eq:centr} of $c(e)$.
\end{proof}

It is interesting to observe that the edges close to the centre are those having the highest importance and that the importance decreases monotonically to its minimum value  moving towards the extreme edges of the graph.

By using the same approach we can provide an explicit expression for  Kemeny's constant of  the graph obtained by adding a loop at vertex $k$ in the one-path graph associated with the $n\times n$ adjacency tridiagonal matrix trid$(1,0,1)$. In fact, we have the following

\begin{theorem}\label{th:add}
    Let $A_n$ be the $n\times n$ adjacency tridiagonal matrix corresponding to a one-path graph with a loop at vertex $k$, i.e., $A_n=\hbox{\rm trid}(1,0,1)+\be_k \be_k^\top$. Then $\bd=A_n\ones=2\cdot\ones-\be_1-\be_n+\be_k$, $\pig=\bd/(2n-1)$ so that \eqref{eq:kqbd} yields
    \[
      \kappa(A_n)=(2n^3-3n^2+(7-6k)n+6k(k-1))/(6n-3).
    \]
\end{theorem}

\section{Centrality of edges in trees formed by three branches}\label{sec:2paths}

Consider the graph depicted in Figure \ref{fig:2g}, which  generalizes the graphs $E_{n,m}$ introduced in \cite{KD} and the one-path graphs of Section \ref{sec:barbell}. The graph is a tree with root in vertex $1$, formed by three branches of lengths $p$, $q$, and $r$, respectively. We refer to this graph as $E_{p,q,r}$. Moreover, we denote by $E_{p,q,r,k}$ the graph obtained from $E_{p,q,r}$ by adding a loop at vertex $k$. We also denote by $F_{n,k}$ the one-path graphs formed by $n$ vertices with a loop at vertex $k$.

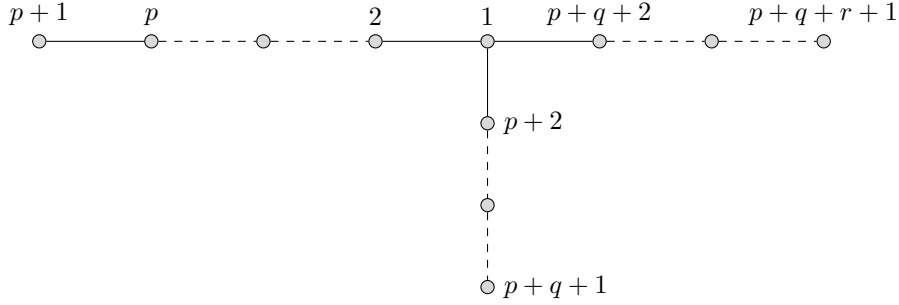
\begin{figure}
    \centering
\begin{tikzpicture}[
      mycircle/.style={
         circle,
         draw=black,
         fill=gray,
         fill opacity = 0.3,
         text opacity=1,
         inner sep=0pt,
         minimum size=5pt,
         font=\small},
      myarrow/.style={},  
     node distance=0.9cm and 1.3cm
      ]
      \node[mycircle] (c1) [label=above:$p+1$]{};
      \node[mycircle,right=of c1] (c2) [label=above:$p$]{};
      \node[mycircle,right=of c2] (c3) {};
      \node[mycircle,right=of c3] (c4) [label=above:$2$]{};
      \node[mycircle,right=of c4] (c5) [label=above:$1$]{};
      \node[mycircle,right=of c5] (c9) [label=above:$p+q+2$]{};
      \node[mycircle,right=of c9] (c11) {};
      \node[mycircle,right=of c11] (c12)[label=above:$p+q+r+1$]{};
      \node[mycircle,below=of c5] (c6) [label=right:$p+2$]{};
      \node[mycircle,below=of c6] (c7) {};
      \node[mycircle,below=of c7] (c8)  [label=right:$p+q+1$]{};

   \draw[myarrow] (c1) -- (c2);
   \draw[myarrow] (c4) -- (c5);
   \draw[style=dashed] (c2) -- (c3);
   \draw[style=dashed] (c3) -- (c4);
  \draw[myarrow] (c5) -- (c9);
   \draw[myarrow] (c5) -- (c6);
   \draw[style=dashed] (c6) -- (c7);
   \draw[style=dashed] (c7) -- (c8);
   \draw[style=dashed] (c9) -- (c11);
   \draw[style=dashed] (c11) -- (c12);
     
    \end{tikzpicture}
\caption{Graph of kind $E_{p,q,r}$ formed by 3 branches.}
    \label{fig:2g}
\end{figure}

Observe that removing any edge from the graph $E_{p,q,r}$, and adding two loops at the vertices of the removed edge, yields two disjoint graphs where
\begin{itemize}
    \item either both graphs are formed by a single path (if the removed edge has $1$ as a vertex), one with a loop at a terminal vertex, the other one with a loop at the root vertex $1$; 
    \item or only one graph is formed by a single path with a loop at a terminal vertex, and the other graph is of the kind $E_{p',q',r',k}$ with suitable values of  $p'$, $q'$, $r'$ and with a loop at a terminal vertex $k$.
\end{itemize} 
Without loss of generality, we may assume that the removed edge is $(i,i+1)$ for $2\le i\le p$, that is, $i$ belongs to the first branch. The other cases can be treated by interchanging the roles of $p,q$, and $r$.

More precisely, we have the following result.
\begin{theorem}\label{th:cases}
Denote by $\mathc G_1$ and $\mathc G_2$ the two disjoint graphs obtained by removing the edge $(i,i+1)$ to the graph of Figure \ref{fig:2g} and adding a loop at vertex $i$ and at vertex $i+1$. We have the following cases. 
\begin{itemize}
    \item If $1< i\le p$, then $\mathc G_1 = F_{p-i+1,1}$,
    $\mathc G_2 = E_{i-1,q,r,i}$;
    \item if $i=1$, then
    $\mathc G_1 = F_{p,1}$,
    $\mathc G_2 = F_{q+r+1,q+1}$.
\end{itemize}
\end{theorem}

Therefore,  for providing an explicit expression of  Kemeny's  constant of the graphs $\mathc G_1$ and $\mathc G_2$ 
it is sufficient to have  general expressions for $\kappa(E_{p,q,r,p+1})$ and $\kappa(F_{n,i})$ for $1<i\le p$. 

We rely on Theorem
\ref{th:add} for the latter quantity. Concerning the former quantity, because of  Remark~\ref{rem:alt}
it is sufficient to provide  explicit expressions for
$\ones^\top\Delta \ones$ and $\ones^\top\Delta \be_{p+1}$.

Concerning the graph $E_{p,q,r}$ it is not difficult to verify that 
the distance matrix has the following form
\[
\Delta=\begin{bmatrix}
0&\bu_p^\top&\bu_q^\top&\bu_r^\top\\
\bu_p&T_p&H_{pq}&H_{pr}\\
\bu_q&H_{pq}^\top&T_q&H_{qr}\\
\bu_r&H_{pr}^\top&H_{qr}^\top&T_r
\end{bmatrix}
\]
where, $\bu_n^\top=[1,2,\ldots,n]$,
$T_n$ is the $n\times n$ Toeplitz matrix with entries $t_{i,j}=(|i-j|)$,
while $H_{mn}$ is the $m\times n$ Hankel matrix with entries $i+j$, $i=1,\ldots,m$, $j=1,\ldots,n$.

Denoting $s=p+q+r$, we have
\begin{equation}\label{eq:toep} 
\begin{aligned}
    &\ones_n^\top T_n\ones_n=\frac13 (n^3-n)\\   
    &\ones_{q}^\top  H_{q,p}\ones_p=
    \frac12 pq(q+p+2),\\
    &\ones^\top \Delta\ones=\frac13 s^3+s^2+\frac23 s-2pqr. 
    \end{aligned}
\end{equation}
From the above structural properties,
we have the following result.
\begin{theorem}\label{th:8}
For the graph $E_{p,q,r}$ in Figure \ref{fig:2g} formed by three branches we have
\[
\bd^\top \Delta \bd=\frac43 s^3+\frac23 s-8pqr,\quad s=p+q+r.
\]
\end{theorem}
\begin{proof}
        The expression follows from \eqref{eq:toep} and Remark \ref{rem:bapat}.
\end{proof}

Concerning the quantity $\bd^\top \Delta \be_{p+1}$, it is easy to verify that 
\begin{equation}\label{eq:15}
    \bd^\top \Delta \be_{p+1}=s^2-2qr.
\end{equation}

The following theorem summarizes the above results. 

\begin{theorem}\label{th:9}
 Kemeny's constant for the graph $E_{p,q,r}$ is given by
\[
\kappa(E_{p,q,r})=\frac16(2s^2+1)-\frac{2pqr}s.
\]
     Kemeny's constant for the graph $E_{p,q,r,p+1}$ is given by
    \[
    \kappa(E_{p,q,r,p+1})=\frac13 s(s+1)-2qr\frac{2p+1}{2s+1}.
    \]
    \end{theorem}
    \begin{proof}
        The expressions follow from Theorem \ref{th:8}, Corollary \ref{cor:2} and \eqref{eq:15}.
    \end{proof}

Combining the above theorem and Corollary \ref{cor:2} we arrive at the following result.

\begin{theorem}\label{th:ci}
The centrality measure $c(i)$ of the edge $(i,i+1)$ of the graph $E_{p,q,r}$ in Figure \ref{fig:2g} for $1\le i\le p$ is
\[
c(i)=
\frac23\alpha_i\beta_i+2qr\frac{2i-1}{2i-1+2q+2r}-\gamma,
\]
where $\alpha_i=p+1-i$ and $\beta_i=q+r+i$ are the number of vertices of the graphs $\mathc G_1$ and $\mathc G_2$, respectively, while $\gamma=\frac13 s-\frac16+\frac{2pqr}s$ is independent of the edge $(i,i+1)$.
\end{theorem}
\begin{proof}
From the definition of centrality given in \eqref{eq:centr}, and from Theorem  \ref{th:cases}, for $1\le i\le p$ we have 
\[
c(i)=\kappa(E_{p,q,r})-\kappa(E_{i-1,q,r,i})-\kappa(F_{p-i+1,1}).
\]
Applying Theorem \ref{th:9} yields
\[
c(i)=\frac16(2s^2+1)-\frac{2pqr}{s}-
\frac13 \hat s(\hat s+1)+2qr\frac{2\hat p+1}{2\hat s+1}-\frac13 \tilde p(\tilde p-1),
\]
where $\tilde p=p-i+1$, 
$\hat p=i-1$, $\hat s=i-1+q+r$.
Substituting $q+r+i=\beta$ and $i=p+1-\alpha$ in the above expression yields the result. 
\end{proof}

Figure \ref{fig:plot} shows a plot of the function $c(i)$ for different values of $p,q,r$. After a possible maximum point close to the value of $i$ for which $\alpha_i=\beta_i$, i.e., $\mathc G_1$ and $\mathc G_2$ have the same cardinality, the function decreases and takes its minimum value at~$i=p$.

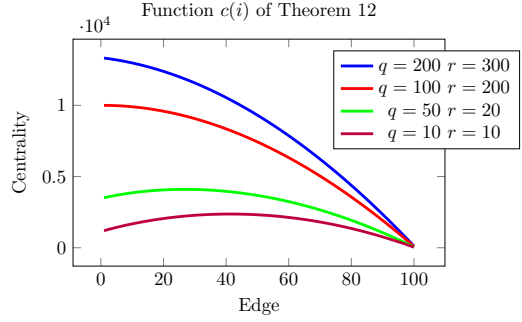
\begin{figure}\label{fig:3}
        \begin{center}               
        \resizebox{7cm}{!}{
           \begin{tikzpicture}
              \begin{axis}[
              legend style={at={(0.7,0.95)},anchor=north west},
                 width = .7\linewidth, height = .3\textheight,
                 xlabel = {Edge}, 
                 ylabel = {Centrality}, title = {Function $c(i)$ of Theorem \ref{th:ci} }]
            \addplot[mark=none,color=blue,line width=1.5pt] table {fig3.dat}; 
            \addplot[mark=none,color=red,line width=1.5pt] table[x index = 0, y index = 2] {fig3.dat};
            \addplot[mark=none,color=green,line width=1.5pt] table[x index = 0, y index = 3] {fig3.dat};
            \addplot[mark=none,color=purple,line width=1.5pt] table[x index = 0, y index = 4] {fig3.dat};
\legend{$q=200 ~ r=300$, 
        $q=100 ~ r=200$,
        $q=50 ~ r=20$,
        $q=10 ~ r=10$,}
        \end{axis}
        \end{tikzpicture}
        }
        \end{center}\caption{Plot of the function $c(i)$ for $p=100$ and for different values of 
        $q$ and $r$ with $i$ in the range $[1, p]$. 
        }\label{fig:plot}
        \end{figure}

\subsection{A generalization}\label{sec:cross}

A similar analysis can be performed with a more general class of tree graphs, i.e. the one formed by trees with $k>3$ branches. For instance, we may consider a tree graph formed by 4 branches of lengths $p,q,r,s$, respectively, as the one depicted in Figure \ref{fig:gen2}.

It can be easily verified that the adjacency matrix has the form
\[
A=\begin{bmatrix}
0&\be_{1,p}^\top &\be_{1,q}^\top &\be_{1,r}^\top &\be_{1,s}^\top \\
\be_{1,p}&W_{p}&0&0&0\\
\be_{1,q}&0&W_{q}&0&0\\
\be_{1,r}&0&0&W_{r}&0\\
\be_{1,s}&0&0&0&W_s\end{bmatrix}
\]
and the distance matrix is given by
\[
\Delta=\begin{bmatrix}
0  &\bu_p^\top &\bu_q^\top &\bu_r^\top &\bu_s^\top \\
\bu_p&T_{p}&H_{pq}&H_{pr}&H_{ps}\\
\bu_q&H_{pq}^\top &T_{q}&H_{qr}&H_{qs}\\
\bu_r&H_{pr}^\top &H_{qr}^\top &T_{r}&H_{rs}\\
\bu_s&H_{ps}^\top &H_{qs}^\top&H_{rs}^\top &T_s
\end{bmatrix}.
\]

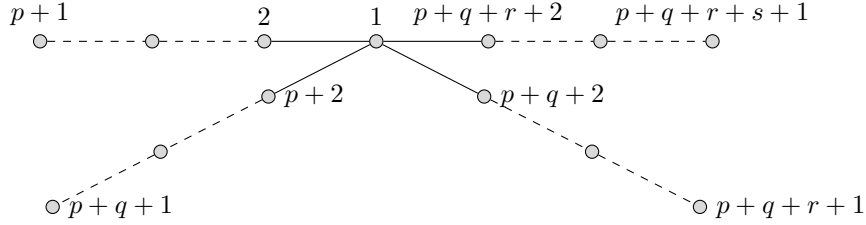
\begin{figure}
    \centering
\begin{tikzpicture}[
      mycircle/.style={
         circle,
         draw=black,
         fill=gray,
         fill opacity = 0.3,
         text opacity=1,
         inner sep=0pt,
         minimum size=5pt,
         font=\small},
      myarrow/.style={},  
     node distance=0.6cm and 1.3cm
      ]
      \node[mycircle] (c1) [label=above:$p+1$]{};
      \node[mycircle,right=of c1] (c3) {};
      \node[mycircle,right=of c3] (c4) [label=above:$2$]{};
      \node[mycircle,right=of c4] (c5) [label=above:$1$]{};
      \node[mycircle,right=of c5] (c9) [label=above:$p+q+r+2$]{};
      \node[mycircle,right=of c9] (c11) {};
      \node[mycircle,right=of c11] (c12)[label=above:$p+q+r+s+1$]{};
      \node[mycircle,below left=of c5] (c6) [label=right:$p+2$]{};
      \node[mycircle,below left=of c6] (c7) {};
      \node[mycircle,below left=of c7] (c8)  [label=right:$p+q+1$]{};
       \node[mycircle,below right=of c5] (r1) [label=right:$p+q+2$]{};
      \node[mycircle,below right=of r1] (r2) {};
      \node[mycircle,below right=of r2] (r3)  [label=right:$p+q+r+1$]{};

   \draw[myarrow] (c4) -- (c5);
   \draw[style=dashed] (c1) -- (c3);
   \draw[style=dashed] (c3) -- (c4);
  \draw[myarrow] (c5) -- (c9);
   \draw[myarrow] (c5) -- (c6);
   \draw[style=dashed] (c6) -- (c7);
   \draw[style=dashed] (c7) -- (c8);
   \draw[style=dashed] (c9) -- (c11);
   \draw[style=dashed] (c11) -- (c12);
     \draw[myarrow] (c5) -- (r1);
     \draw[style=dashed] (r1) -- (r2);
     \draw[style=dashed] (r2) -- (r3);
    \end{tikzpicture}

    \caption{Tree graph formed by four branches.}
    \label{fig:gen2}
\end{figure}

The structure of the above matrices can be easily extended to the case of any number of $k$ branches.
Relying on these expressions we may arrive at explicit formulas for the centrality measure of these classes of graphs. We omit the details that are technical and do not add much more information.

\section{Some numerical experiments}\label{sec:exp}
In this section, we report on some numerical experiments performed in Matlab v.~9.5.0 (2018b) on a PC with Intel I3 processor and Ubuntu operating system Release 20.04.6 LTS. 

To verify the effectiveness of equation \eqref{eq:centr} for computing the centrality of a cut-edge, we have compared the result provided by \eqref{eq:centr} to the result obtained by applying the original definition given in \cite{abcmp}, that is equation \eqref{eq:cer}. 
Here, the computation of  $c_r(e)$ in \eqref{eq:cer} has been performed with the software \url{https://github.com/numpi/kemeny-based-centrality} designed in \cite{abcmp},
while the implementation of \eqref{eq:centr} has been performed relying on \eqref{eq:keminv}.  In the case of networks of very large size, for the computation of Kemeny's constants it is convenient to use the algorithm of \cite{bdkm} that relies on Schur complements.
We have considered a set of synthetic graphs, formed by one-path graphs, barbell graphs, binary trees, and star-shaped  graphs; and a real-world graph given by the road map of Pisa (Italy).

\subsection{Synthetic data}
Consider a binary tree graph with $2^m-1$ nodes and a one-path graph. For all the edges $e$ in the graph, we have computed the values of $c(e)$ in high precision by using the multi-precision toolbox Advanpix (\url{https://www.advanpix.com/}), together with the values of $c_r(e)$ and $c(e)$ computed in standard precision. We have set $r=1-s$ and chosen values of $s=10^{-i}$, $i=3,\ldots,8$.

  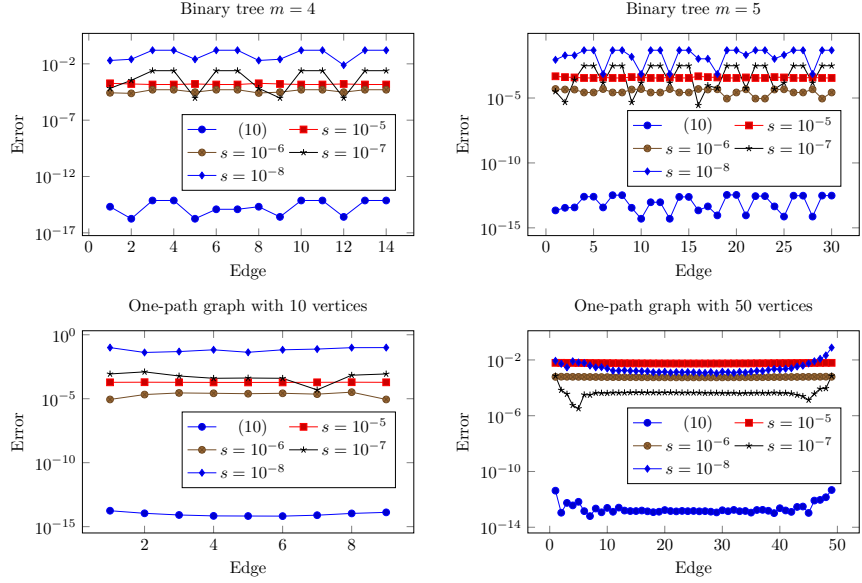
\begin{figure}
        \begin{center}
                \begin{tabular}{cc}
                \resizebox{5.5cm}{!}{
        \begin{tikzpicture}
        \begin{semilogyaxis}[
                legend style={at={(0.3,0.6)},anchor=north west},
                legend columns=2,
                width = .7\linewidth, height = .3\textheight,
                xlabel = {Edge}, 
                ylabel = {Error}, title = {Binary tree $m = 4$}]
            \addplot table {err.txt}; 
            \addplot table[x index = 0, y index = 4] {err.txt};
            \addplot table[x index = 0, y index = 5] {err.txt};
            \addplot table[x index = 0, y index = 6] {err.txt};
            \addplot table[x index = 0, y index = 7] {err.txt};
\legend{\eqref{eq:centr}, 
$s=10^{-5}$,$s=10^{-6}$,$s=10^{-7}$,$s=10^{-8}$}
        \end{semilogyaxis}
        \end{tikzpicture}}
&       
                \resizebox{5.5cm}{!}{
        \begin{tikzpicture}
        \begin{semilogyaxis}[
                legend style={at={(0.3,0.62)},anchor=north west},
                legend columns=2,
                width = .7\linewidth, height = .3\textheight,
                xlabel = {Edge}, 
                ylabel = {Error}, title = {Binary tree $m=5$}]
            \addplot table {err1.txt}; 
            \addplot table[x index = 0, y index = 4] {err1.txt};
            \addplot table[x index = 0, y index = 5] {err1.txt};
            \addplot table[x index = 0, y index = 6] {err1.txt};
            \addplot table[x index = 0, y index = 7] {err1.txt};
\legend{\eqref{eq:centr}, 
$s=10^{-5}$,$s=10^{-6}$,$s=10^{-7}$,$s=10^{-8}$}
        \end{semilogyaxis}
        \end{tikzpicture}}\\
                        \resizebox{5.5cm}{!}{
        \begin{tikzpicture}
        \begin{semilogyaxis}[
                legend style={at={(0.3,0.6)},anchor=north west},
                legend columns=2,
                width = .7\linewidth, height = .3\textheight,
                xlabel = {Edge}, 
                ylabel = {Error}, title = {One-path graph with 10 vertices}]
            \addplot table {err2.txt}; 
            \addplot table[x index = 0, y index = 4] {err2.txt};
            \addplot table[x index = 0, y index = 5] {err2.txt};
            \addplot table[x index = 0, y index = 6] {err2.txt};
            \addplot table[x index = 0, y index = 7] {err2.txt};
\legend{\eqref{eq:centr}, 
$s=10^{-5}$,$s=10^{-6}$,$s=10^{-7}$,$s=10^{-8}$}
        \end{semilogyaxis}
        \end{tikzpicture}}
        &
                        \resizebox{5.5cm}{!}{
        \begin{tikzpicture}
        \begin{semilogyaxis}[
                legend style={at={(0.3,0.62)},anchor=north west},
                legend columns=2,
                width = .7\linewidth, height = .3\textheight,
                xlabel = {Edge}, 
                ylabel = {Error}, title = {One-path graph with  50 vertices}]
            \addplot table {err3.txt}; 
            \addplot table[x index = 0, y index = 4] {err3.txt};
            \addplot table[x index = 0, y index = 5] {err3.txt};
            \addplot table[x index = 0, y index = 6] {err3.txt};
            \addplot table[x index = 0, y index = 7] {err3.txt};
\legend{\eqref{eq:centr}, 
$s=10^{-5}$,$s=10^{-6}$,$s=10^{-7}$,$s=10^{-8}$}
        \end{semilogyaxis}
        \end{tikzpicture}}

\end{tabular}
\end{center}\caption{Relative errors in computing $c(e)$ by means of equation \eqref{eq:centr} and as $\lim_{r\to 1} c_r(e)$ by means of \eqref{eq:cer} for several values of $r=1-s$. In the first line, we have the case of a binary tree of depth 4 (left) and of depth 5 (right).
    In the second line, we have the case of a one-path graph formed by 10 vertices (left) and by 50 vertices (right). Small values of $s$ produce large cancellation errors, and large values produce a poor approximation of the limit.}\label{fig:err}
\end{figure}

The relative errors with respect to the high precision values, assumed as reference values, are plotted in Figure \ref{fig:err} with $m=4$ and $m=5$ for the binary tree (upper plots), and for one-path graphs formed by 10 and by 50 vertices, respectively (lower plots). It is interesting to point out that small values of $s$ would provide values of $c_r(e)$ closer to the sought limit value  $\lim_{r\to 1}c_r(e)$. On the other hand, the smaller $s$, the higher the numerical cancellation errors in the floating-point  computation of \eqref{eq:cer}. Observe that the value of $c_r(e)$ obtained in \eqref{eq:cer} is the result of a subtraction of two terms that tend to infinity as $r$ tends to 1. This explains why the error is higher the closer the value of $r$ to 1.  This  source of cancellation appears clearly from the graphs in Figure~\ref{fig:err} where the most accurate floating-point approximation, with a relative error of about $10^{-5}$, is obtained for $s=10^{-6}$, or for $s=10^{-7}$, whereas the error provided by \eqref{eq:centr} is much closer to the machine precision. This confirms the better numerical stability of this computation which overcomes numerical cancellation. 

By looking at Figure \ref{fig:err}, we may observe that the rounding errors, in the case of binary trees, have an oscillatory behaviour with respect to the edge index of the binary tree. This feature likely depends on the numeration of the edges and on the fact that edges at the same level in the tree are affected by the same rounding errors.

A second set of experiments aims to confirm that the centrality of a cut-edge connecting two subgraphs formed by $m_1$ and $m_2$ vertices is higher the closer to 1 is the ratio $m_1/m_2$.

\begin{figure}
    \centering
    \includegraphics[width=0.45\linewidth]{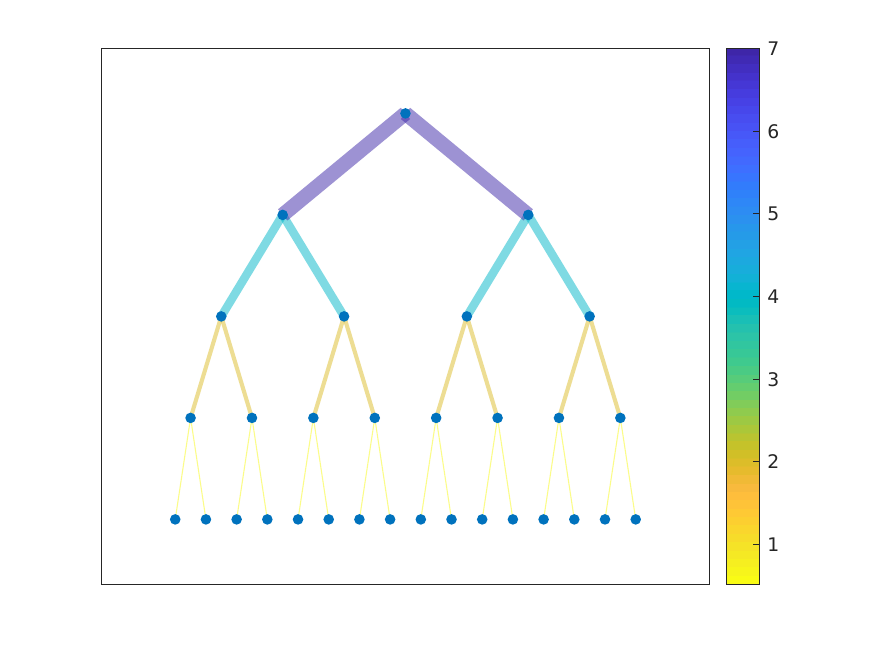} \includegraphics[width=0.45\linewidth]{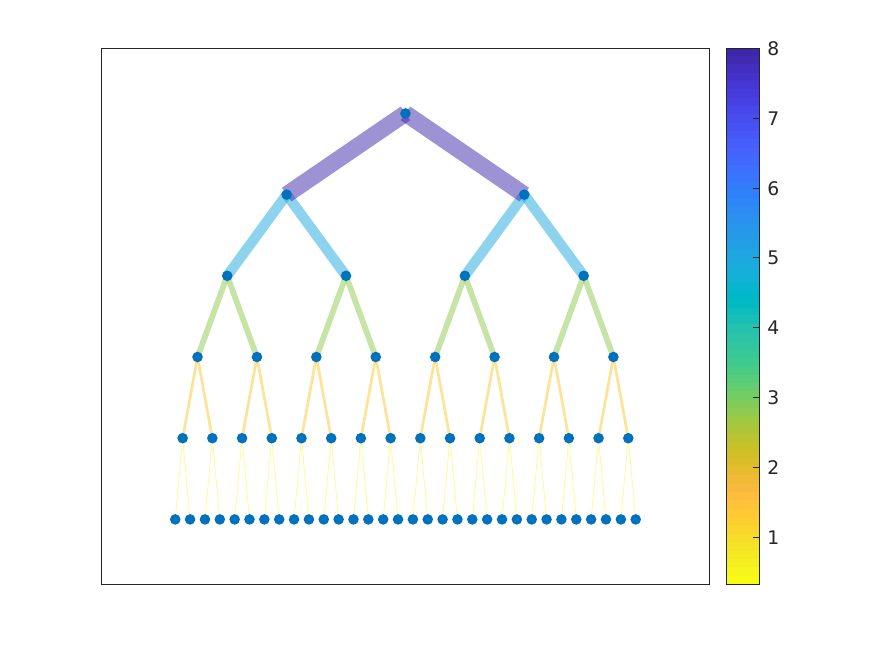}
    \caption{Binary tree graph with $2^m-1$ nodes, where edges are displayed with different thicknesses and colours according to their centrality. On the left and the right, the graphs obtained with $m=5$ and $m=6$, respectively.   Observe that the centrality is constant for each level.}
    \label{fig:bintree}
\end{figure}

\begin{figure}
    \centering
    \includegraphics[width=0.45\linewidth]{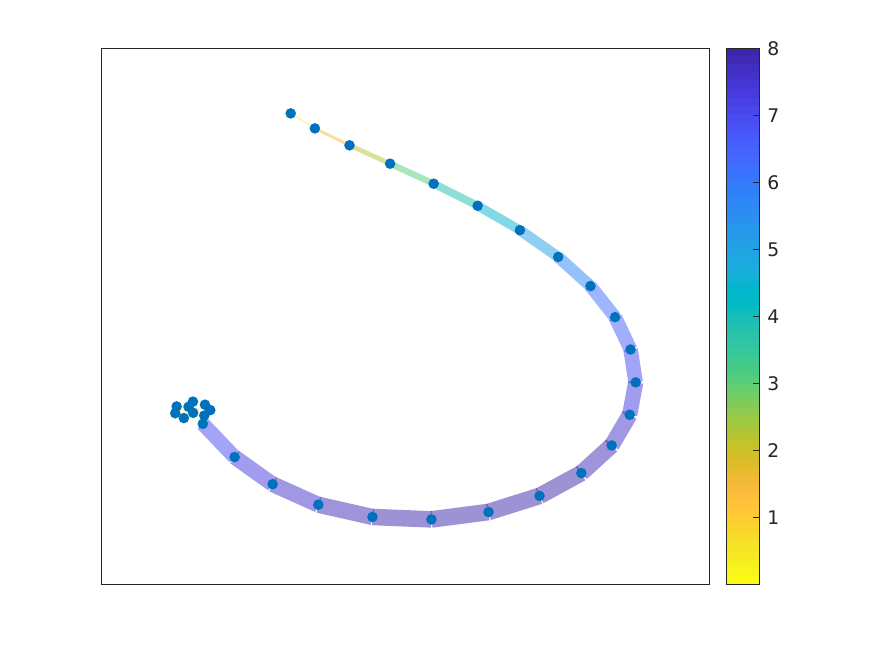}
     \includegraphics[width=0.45\linewidth]{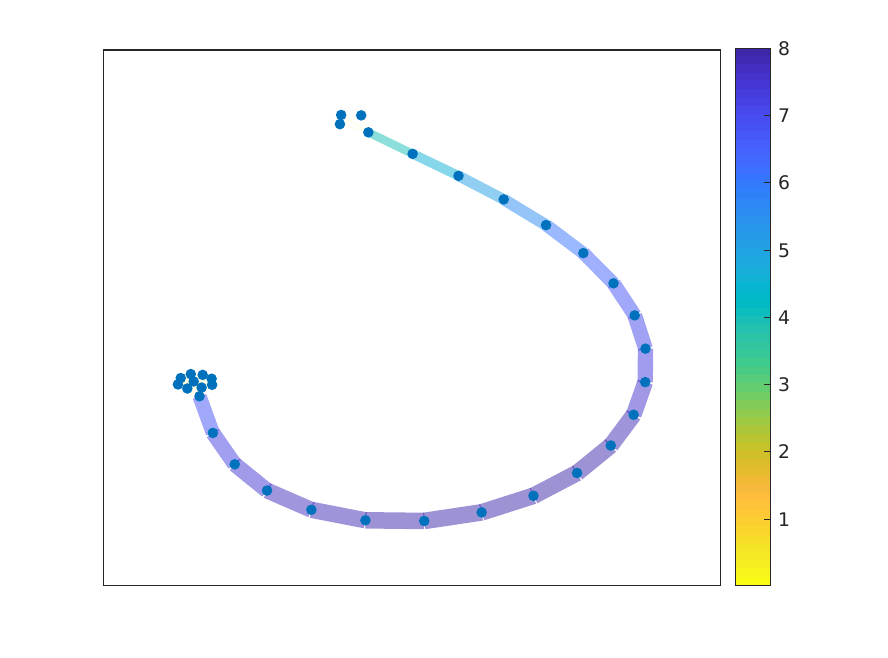}
    \caption{Barbell graphs: Two complete graphs with $m$ and $n$ nodes are connected by a path formed by $p$ edges. On the left, the case $p=20$, $m=8$, $n=2$; on the right the case $p=20$, $m=8$, $n=4$.}
    \label{fig:barbell}
\end{figure}

\begin{figure}
    \centering
    \includegraphics[width=0.45\linewidth]{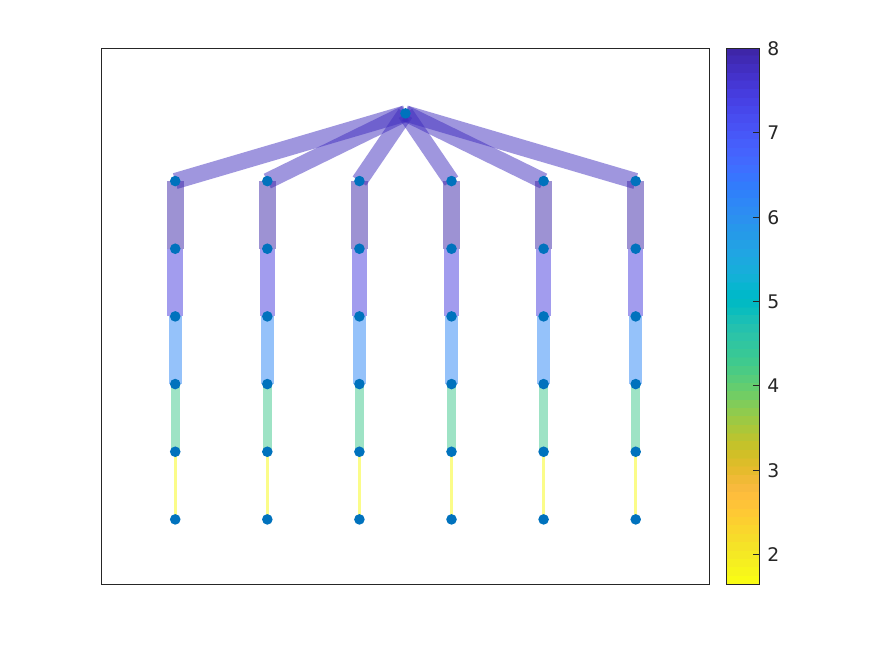}     \includegraphics[width=0.45\linewidth]{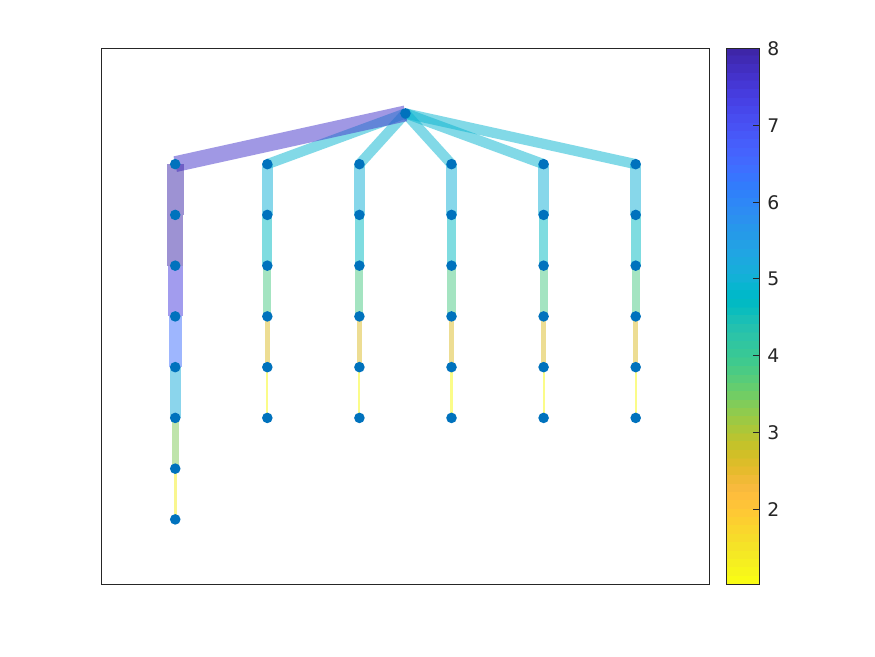}
    \caption{Star-shaped graph: The variation of the length of a path causes a substantial change in the centrality score of cut edges.}
    \label{fig:star}
\end{figure}

To this end, besides binary trees, we have considered barbell graphs formed by two complete graphs connected by a single path (Figure \ref{fig:barbell}),  star-shaped graphs (Figure~\ref{fig:star}), and one-path graphs where the vertices are replaced by cliques (Figure~\ref{fig:barkron}).

We have verified that in all these cases the centrality of a cut-edge $e$ is higher when the ratio of the sizes of the subgraphs connected by $e$ is closer to 1.

In all the figures, all the graphs are plotted so that 
each edge's thickness and colour is proportional to the centrality of the edge. We may observe that for the binary tree in Figure \ref{fig:bintree}, the centrality decreases with the level of the edge. The first two edges connected to the root have the highest value of centrality. These edges connect two subgraphs of similar cardinality. 
The minimum centrality is taken by the edges in the last level that connect a single vertex to a subgraph of large cardinality.

\begin{figure}
    \centering
    \includegraphics[width=0.45\linewidth]{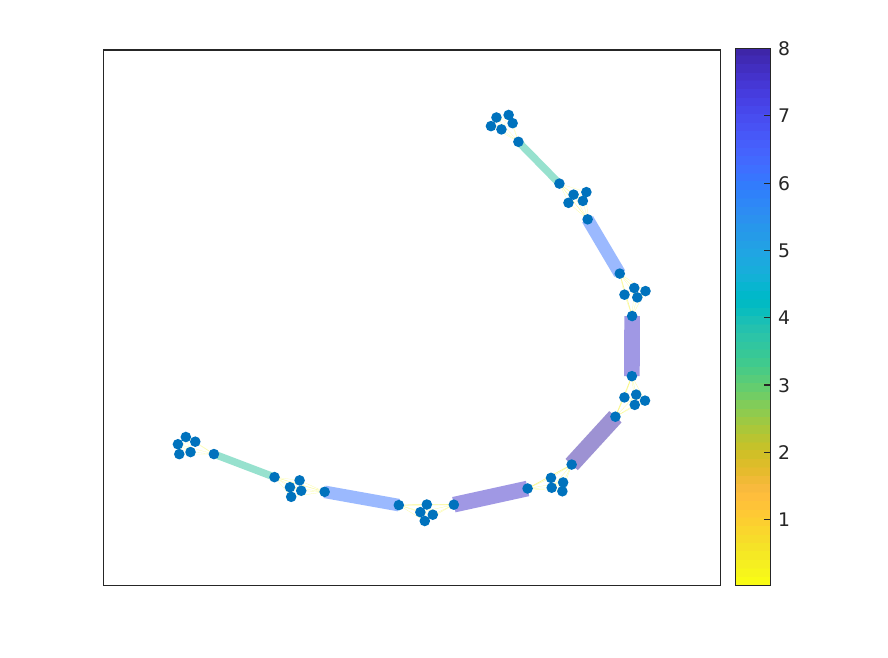}     \includegraphics[width=0.45\linewidth]{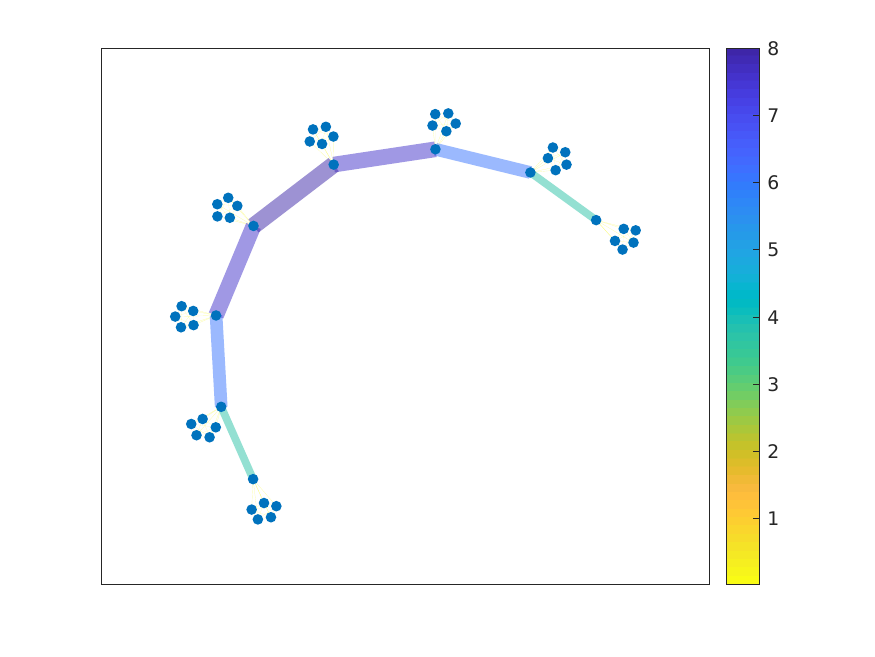}
    \caption{One-path graphs where cliques replace the vertices with two different kinds of connections. The central edges maintain a larger centrality score.}
    \label{fig:barkron}
\end{figure}

The numerical values of the centrality scores of each level are reported in Table~\ref{tab:1} where $m$ denotes the depth of the binary tree.

\begin{table} 
    \centering
    \begin{tabular}{c|cccccc}
     $m~\backslash$ Level    & 1&2&3&4&5&6 \\ \hline
      2   
& 0.8333
      \\ 3
&4.7571 
&1.5
\\ 4 
&16.0919 
&8.1236 
&2.2936 
\\ 5
&42.9954 
&24.8283
&12.0818
&3.1677
\\      6
&101.5850
&62.4365
&34.9999
&16.4139
&4.0937
\\7
&223.9567
&142.3938
&84.8946
&46.1332
&20.9977
&5.0515
    \end{tabular}
    \caption{Numerical values of the centrality scores of the edges of binary graphs of $2^m-1$ vertices.}
    \label{tab:1}
\end{table}

A similar situation holds for the barbell graphs in Figure \ref{fig:barbell} where edges closer to one of the two end points have a lower centrality score and edges that split the graph into two subgraphs of almost the same size have the highest score. The size of the two cliques influences the score of the cut-edges.

An analogous situation is displayed in Figure \ref{fig:star} where a star-shaped graph is displayed. In the graph on the left, all the paths have the same length and the centrality scores of each branch are the same and decrease if the edge gets closer to the endpoint. In the graph on the right, the length of the first branch has been increased. This leads to the reduction of the score of the edges in the other branches.

Figure \ref{fig:barkron} reports two one-path graphs where each vertex has been replaced by a clique and different kinds of connections are applied. Once again, in both cases, the central edges have a higher centrality score.

\subsection{The case of a road map}
In this section, we performed some numerical experiments on Pisa's road map. In this map, shown in Figure \ref{fig:pisa}, there are 1404 vertices and 3588 edges, among which there are 221 cut-edges, and they are represented with a thicker blue line in the figure. 

\begin{figure}
    \centering
    \includegraphics[width=0.34\linewidth]{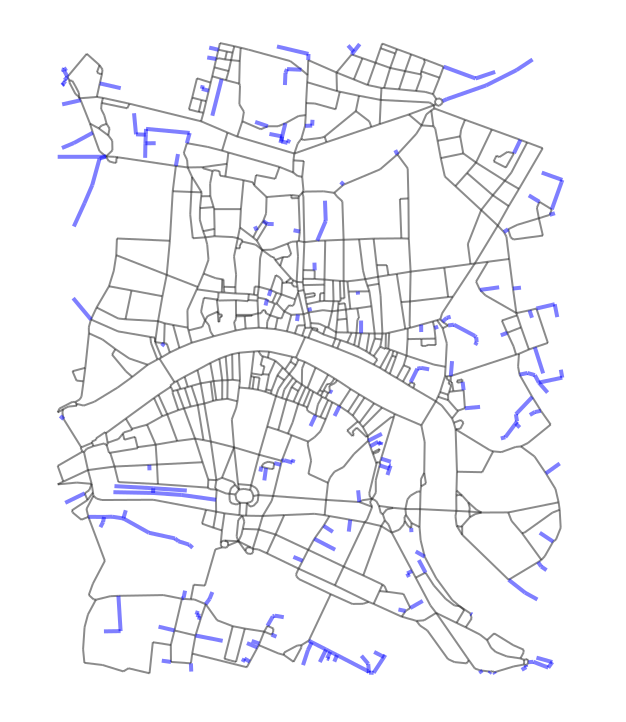} ~~
    \raisebox{0.33cm}{\includegraphics[width=0.294\linewidth]{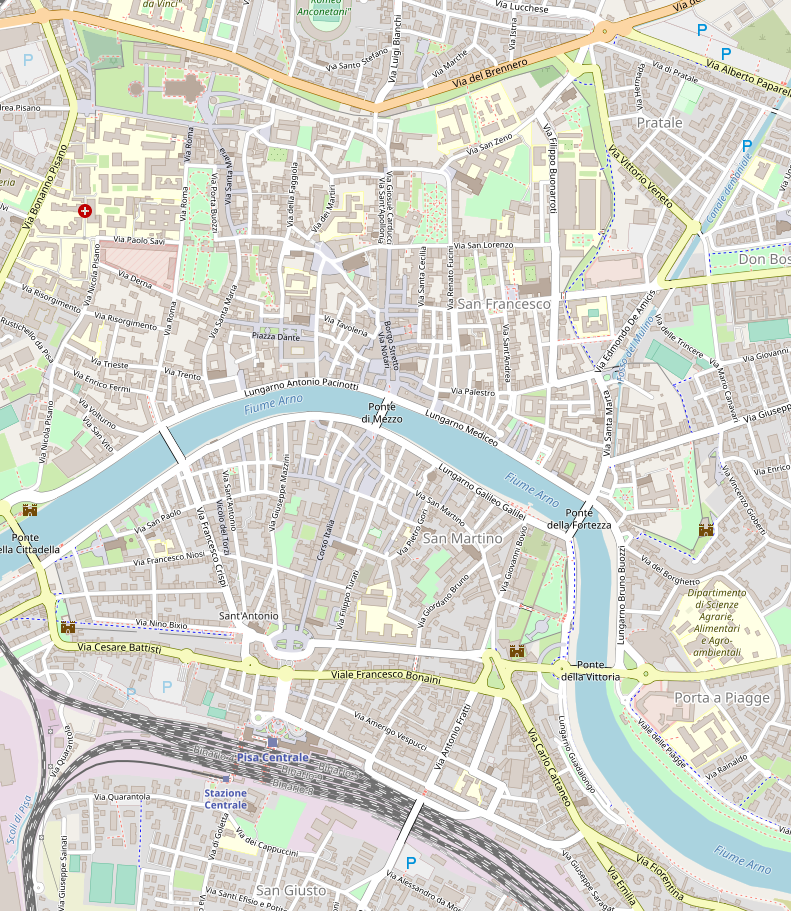}}
    \caption{{Road map of Pisa. This graph has 221 cut-edges that are represented with a thicker blue line.}}
    \label{fig:pisa}
\end{figure}

We have modified the software developed in \cite{abcmp} in order to provide in output the indices of all the cut-edges according to the following criterion: $e$ is a cut-edge if $c_r(e)>\frac1{2(1-r)}$. The centralities of the cut-edges are computed by means of the function {\tt kementrality} of the package
\url{https://github.com/numpi/kementrality/archive/refs/heads/main.zip}
with several values of the regularization parameter $r=1-s$, i.e., for $s=10^{-j}$, $j=5,6,7,8$. The same values are computed with equation \eqref{eq:centr} in double precision and extended precision by means of the toolbox {\tt Advanpix}. These latter values are the reference values for computing the relative errors. Figure \ref{fig:errpisa} reports these  errors for the different values of $s$ and for the new formula.

\begin{figure}
       \begin{center}
       \begin{tabular}{cc}
        \resizebox{5.5cm}{!}{
        \begin{tikzpicture}
        \begin{semilogyaxis}[
                legend style={at={(0.3,0.6)},anchor=north west},
                legend columns=2,
                width = .7\linewidth, height = .3\textheight,
                xlabel = {Cut-edge}, 
                ylabel = {Error}, title = {Pisa road map, cut-edges 1--60}]
            \addplot table {err_pisa1.txt};
            \addplot table[x index = 0, y index = 3] {err_pisa1.txt};
            \addplot table[x index = 0, y index = 4] {err_pisa1.txt};
            \addplot table[x index = 0, y index = 5] {err_pisa1.txt};
            \legend{\eqref{eq:centr},
               $s=10^{-5}$,
               $s=10^{-6}$,
               $s=10^{-7}$}
        \end{semilogyaxis}
        \end{tikzpicture}
        }
        &
                \resizebox{5.5cm}{!}{
        \begin{tikzpicture}
        \begin{semilogyaxis}[
                legend style={at={(0.3,0.6)},anchor=north west},
                legend columns=2,
                width = .7\linewidth, height = .3\textheight,
                xlabel = {Cut-edge}, 
                ylabel = {Error}, title = {Pisa road map, cut-edges 61--120}]
            \addplot table {err_pisa2.txt};
            \addplot table[x index = 0, y index = 3] {err_pisa2.txt};
            \addplot table[x index = 0, y index = 4] {err_pisa2.txt};
            \addplot table[x index = 0, y index = 5] {err_pisa2.txt};
            \legend{\eqref{eq:centr},
               $s=10^{-5}$,
               $s=10^{-6}$,
               $s=10^{-7}$}
        \end{semilogyaxis}
        \end{tikzpicture}
        }\\
                \resizebox{5.5cm}{!}{
        \begin{tikzpicture}
        \begin{semilogyaxis}[
                legend style={at={(0.3,0.6)},anchor=north west},
                legend columns=2,
                width = .7\linewidth, height = .3\textheight,
                xlabel = {Cut-edge}, 
                ylabel = {Error}, title = {Pisa road map, cut-edges 121--180}]
            \addplot table {err_pisa3.txt};
            \addplot table[x index = 0, y index = 3] {err_pisa3.txt};
            \addplot table[x index = 0, y index = 4] {err_pisa3.txt};
            \addplot table[x index = 0, y index = 5] {err_pisa3.txt};
            \legend{\eqref{eq:centr},
               $s=10^{-5}$,
               $s=10^{-6}$,
               $s=10^{-7}$}
        \end{semilogyaxis}
        \end{tikzpicture}
        }&
                \resizebox{5.5cm}{!}{
        \begin{tikzpicture}
        \begin{semilogyaxis}[
                legend style={at={(0.3,0.6)},anchor=north west},
                legend columns=2,
                width = .7\linewidth, height = .3\textheight,
                xlabel = {Cut-edge}, 
                ylabel = {Error}, title = {Pisa road map, cut-edges 181--221}]
            \addplot table {err_pisa4.txt};
            \addplot table[x index = 0, y index = 3] {err_pisa4.txt};
            \addplot table[x index = 0, y index = 4] {err_pisa4.txt};
            \addplot table[x index = 0, y index = 5] {err_pisa4.txt};
            \legend{\eqref{eq:centr},
               $s=10^{-5}$,
               $s=10^{-6}$,
               $s=10^{-7}$}
        \end{semilogyaxis}
        \end{tikzpicture}
        }
        \end{tabular}
\end{center}
\caption{Relative errors in computing $c(e)$, where $e$ are the cut-edges in the road map of Pisa. The computation is performed  by means of equation \eqref{eq:centr} and as $\lim_{r\to 1} c_r(e)$ by means of \eqref{eq:cer} for several values of $r=1-s$. In the x-axis the indices of the cut edges, in the y-axis the relative errors.}\label{fig:errpisa}
\end{figure}

We may observe that, while for the formula \eqref{eq:centr} the errors are roughly in between $10^{-13}$ and $10^{-10}$, for the computation by means of \eqref{eq:cer} with $s=6$, the errors are in the range $[10^{-4},10^{-3}]$. Worse bounds are obtained for different values of $s$. In Table~\ref{tab:range} we report these bounds (leftmost columns) for different values of $s$ together with the corresponding bounds of the one-path graph with 11 vertices (rightmost columns). In the last line, we report the values of the errors obtained by performing the computation by means of \eqref{eq:centr}. We may notice that the relative errors in the case of the road map are sensibly larger than the corresponding errors for the one-path graph.
\begin{table}
    \begin{center}
        \begin{tabular}{c|cc|cc}
        $s$& Minimum error& Maximum error& Minimum error& Maximum error\\ \hline
        $10^{-4}$&
        $1.3\cdot 10^{-2}$ & $8.3\cdot 10^{-2}$
        & $1.9\cdot 10^{-3}$ & $2.0\cdot 10^{-3}$\\
        
        $10^{-5}$&
        $1.4\cdot 10^{-3}$ & $9.5\cdot 10^{-3}$
        & $1.9\cdot 10^{-4}$ & $2.0\cdot 10^{-4}$\\
        
        $10^{-6}$&
        $1.2\cdot 10^{-4}$ & $9.6\cdot 10^{-4}$
        &$8.8\cdot 10^{-6}$&$3.3\cdot 10^{-5}$\\
        
        $10^{-7}$&
        $1.2\cdot 10^{-6}$ & $5.6\cdot 10^{-3}$
        &$4.9\cdot 10^{-5}$& $1.3\cdot 10^{-3}$\\
        
        $10^{-8}$&
        $5.3\cdot 10^{-4}$ & $3.4\cdot 10^{-1}$
        & $4.1\cdot 10^{-2}$& $9.9\cdot 10^{-2}$\\
        $10^{-9}$&$3.4\cdot 10^{-1}$ & $8.4\cdot 10^{1}$
        & $7.0\cdot10^{0}$ &$1.2\cdot 10^1$ \\
        \eqref{eq:centr}&$1.1\cdot 10^{-13}$ &$1.3\cdot 10^{-10}$
        &$6.7\cdot 10^{-15}$ &$1.8\cdot 10^{-14}$
        \end{tabular}
    \end{center}\caption{Range of the relative errors in computing the centrality of cut-edges by means of \eqref{eq:cer} for different values of $s$. The two columns on the left report the values of the road map of Pisa, and the two columns on the right report the values of the one-path graph with 11 vertices. In the last line, the relative errors obtained by applying \eqref{eq:centr}.}\label{tab:range}
\end{table}

\section{Conclusions}\label{sec:conc}
We have revisited the centrality measure $c(e)$ introduced in \cite{abcmp}, for the cut-edges $e$ of a given graph $\mathc G$. An explicit expression of $c(e)$ is given in terms of  Kemeny's constant of $\mathc G$ and two suitable subgraphs. This expression avoids the introduction of a regularization parameter and overcomes the problem of numerical cancellation intrinsic in the approach of \cite{abcmp}. A  physical interpretation of this measure is given in terms of the stochastic complement of the probability matrix associated with the random walk on the graph $\mathc G$ with respect to suitable subgraphs.

This measure has been explicitly expressed for one-path graphs, i.e., associated with a tridiagonal adjacency matrix, and for more general tree graphs formed by 3 branches. A  generalization to the case of $n$ branches has been outlined. These expressions confirm the nice behaviour of this measure observed experimentally in the paper \cite{abcmp}, where the centrality of a cut-edge that connects two subgraphs formed by $m_1$ and $m_2$ vertices is higher the closer to 1 is the ratio $m_1/m_2$. As a side result, a generalization of a result of \cite{kirk:ela} to the case of graphs with loops is given. 
Numerical experiments are shown to verify that the new expression is numerically stable and to confirm the good physical behaviour of this measure.


\section*{Funding}
This work is partially supported by
\begin{itemize}
    \item MUR Excellence Department Project awarded to the Department of Mathematics, University of Pisa, CUP I57G22000700001. 
    \item European Union - NextGenerationEU under the National Recovery and Resilience Plan (PNRR) - Mission 4 Education and research - Component 2 From research to business - Investment 1.1 Notice Prin 2022 - DD N. 104  2/2/2022, titled Low-rank Structures and Numerical Methods in Matrix and Tensor Computations and their Application, proposal code 20227PCCKZ – CUP 
I53D23002280006. 
\item GNCS of INdAM (Istituto Nazionale di Alta Matematica).
\item  Natural Sciences and Engineering
Research Council of Canada, grant number RGPIN–2019–05408.
\end{itemize}

The fourth author is a member of GNCS of INdAM.



\end{document}